\documentclass[a4paper,reqno,10pt]{article}

\usepackage{epsf,graphicx,epsfig}
\usepackage{amscd}
\usepackage{amsmath,latexsym,amssymb,amsthm}
\usepackage[nospace,noadjust]{cite}
\usepackage{textcomp}
\usepackage{setspace,cite}
\usepackage{lscape,fancyhdr,fancybox}
\usepackage{stmaryrd}
\usepackage[all,cmtip]{xy}
\usepackage{tikz}
\usepackage{cancel}
\usetikzlibrary{shapes,arrows,decorations.markings}
\usepackage{graphicx}

\usepackage{color}
\usepackage{url}
\usepackage{enumerate}
\usepackage[mathscr]{euscript}

  \theoremstyle{plain}

\swapnumbers
    \newtheorem{thm}{Theorem}[section]
    \newtheorem{prop}[thm]{Proposition}

    \newtheorem{corollary}[thm]{Corollary}
    
    \newtheorem{subsec}[thm]{}
\theoremstyle{definition}
    \newtheorem{defn}[thm]{Definition}
        \newtheorem{remark}[thm]{Remark}
    \newtheorem{exam}[thm]{Example}

\theoremstyle{remark}

\title{}
\author{}
\date{}
\usepackage{amssymb}

\usepackage{hyperref}
\hypersetup{
	colorlinks,
	citecolor=blue,
	filecolor=black,
	linkcolor=blue,
	urlcolor=black
}

\begin{document}

\title{Cohomology of morphism Lie algebras and some applications}

\author{Apurba Das\footnote{Email: apurbadas348@gmail.com
}}

\maketitle

\begin{center}
\end{center}



\begin{abstract}
A morphism Lie algebra is a triple $(\mathfrak{g}, \mathfrak{h}, \phi)$ consisting of two Lie algebras $\mathfrak{g}, \mathfrak{h}$ and a Lie algebra homomorphism $\phi : \mathfrak{g} \rightarrow \mathfrak{h}$. We define representations and cohomology of morphism Lie algebras.  
As applications of our cohomology, we study some aspects of deformations, abelian extensions of morphism Lie algebras and classify skeletal morphism sh Lie algebras. Finally, we consider the cohomology of morphism Lie groups and find a relation with the cohomology of morphism Lie algebras.
\end{abstract}

\medskip

\noindent {\sf 2010 MSC classification:} 17B55, 17B56.

\noindent {\sf Keywords:} Morphism Lie algebras, Cohomology, Deformations, Morphism sh Lie algebras, Morphism Lie groups. 

\medskip


\thispagestyle{empty}

\tableofcontents



\section{Introduction}
Cohomology of some algebraic structures (e.g. associative, Lie, commutative algebras) plays a key role in the study of deformations and extensions of algebras \cite{gers,hoch}. They are also useful in representation theory \cite{benson}, category theory \cite{van}, differential geometry of Lie groups \cite{che}, homotopy algebras \cite{baez-crans} and quantization of Poisson manifolds in mathematical physics \cite{kont}. In the last twenty years, cohomology theory has been formulated for various algebras including algebras over binary quadratic operads \cite{bala-operad} and combinatorial Loday-type algebras \cite{das-jpaa}. Our main objective in this paper is the notion that consists of a triple $(\mathfrak{g}, \mathfrak{h}, \phi)$ of two Lie algebras $\mathfrak{g}, \mathfrak{h}$ and a Lie algebra homomorphism $\phi : \mathfrak{g} \rightarrow \mathfrak{h}$ between them. We call such a triple a morphism Lie algebra. Thus, a morphism Lie algebra captures information of two Lie algebras as well as a Lie algebra homomorphism between them. Morphism Lie algebras arise from Rota-Baxter operators, Nijenhuis operators and differentiating Lie group homomorphisms. The aim of this paper is to define and study the cohomology theory of morphism Lie algebras. We provide applications of our cohomology in some aspects of deformation problems, abelian extensions and homotopy algebras. We also consider morphism Lie groups, define their cohomology and find a relation with morphism Lie algebras.

\medskip

Deformation theory is not only limited to algebraic structures. In \cite{nij-ric-hom} Nijenhuis and Richardson first considered deformations of algebra homomorphisms by keeping the underlying algebras intact. Later, simultaneous deformations of algebras and homomorphisms were considered by Gerstenhaber and Schack \cite{gers-sch}. Simultaneous deformations of two Lie algebras $\mathfrak{g}, \mathfrak{h}$ and a Lie algebra homomorphism $\phi : \mathfrak{g} \rightarrow \mathfrak{h}$ were studied by Fr\'{e}gier \cite{fregier} by introducing a suitable cohomology. As mentioned before, we call the whole triple $(\mathfrak{g}, \mathfrak{h}, \phi)$ a morphism Lie algebra. We define representations of a morphism Lie algebra $(\mathfrak{g}, \mathfrak{h}, \phi)$ and introduce cohomology with coefficients in a representation. When considering the cohomology with coefficients in the adjoint representation, our cohomology coincides with the one introduced by Fr\'{e}gier. We also find a sufficient condition for the vanishing of our cohomology in terms of the vanishing of some Lie algebra cohomology groups.

\medskip

The cohomology of some type of algebras with coefficients in the adjoint representation only allows the study of deformations of algebras. However, to study some other deformation problems (e.g. deformations of homomorphisms between algebras and deformations of subalgebras), one needs to know the cohomology with coefficients in a representation \cite{nij-ric-hom,ric-sub}. Therefore, our cohomology of morphism Lie algebras with coefficients in a representation gave us the freedom to consider the above-mentioned deformation problems for morphism Lie algebras. First, we study deformations of homomorphisms between morphism Lie algebras. We also find a sufficient condition for the rigidity of a homomorphism between morphism Lie algebras. Second, we study deformations of morphism Lie subalgebras in terms of the cohomology of a suitable morphism Lie algebra. 

\medskip

In the next part, we study abelian extensions of morphism Lie algebras in terms of cohomology. More precisely, we introduce a notion of `simple' cohomology of morphism Lie algebras and show that isomorphism classes of abelian extensions are classified by the simple second cohomology group. Another justification of our cohomology can be given in terms of some suitable morphism sh Lie algebras. Recall that, sh Lie algebras (strongly homotopy Lie algebras) are the homotopy analogue of Lie algebras where the Jacobi identity holds only up to a homotopy. In \cite{baez-crans} Baez and Crans showed that skeletal sh Lie algebras are classified by the third cohomology groups of Lie algebras. We generalize their result in the context of morphism Lie algebras. More precisely, we introduce skeletal morphism sh Lie algebras and show that such structures can be classified by the third cohomology groups of morphism Lie algebras introduced in the paper.

\medskip

Finally, we consider morphism Lie groups, the global object for morphism Lie algebras. We introduce the cohomology of morphism Lie groups with coefficients in a module. This cohomology is obtained as byproducts of known group cohomologies. In the end, we define the differential cohomology of a morphism Lie group and find its relation with the cohomology of morphism Lie algebra.

\medskip

The paper is organized as follows. In Sections \ref{sec-2} and \ref{section-cohomology}, we respectively study representations and cohomology of morphism Lie algebras. In Section \ref{sec-4}, we study deformations of homomorphisms between morphism Lie algebras and deformations of morphism Lie subalgebras in terms of cohomology. Abelian extensions of morphism Lie algebras and classification of skeletal morphism sh Lie algebras are given in Sections \ref{sec-5} and \ref{sec-6}, respectively. Finally, morphism Lie groups and their cohomology theoretic results are considered in Section \ref{sec-7}.

\medskip

All vector spaces, (multi)linear maps, tensor products and wedge products are over a field {\bf k} of characteristic $0$.

\section{Morphism Lie algebras and their representations}\label{sec-2}
In this section, we introduce morphism Lie algebras and define their representations.

Let $\mathfrak{g}$ be a Lie algebra. We denote the Lie bracket on $\mathfrak{g}$ by $[~,~]_\mathfrak{g}$. A representation of the Lie algebra $\mathfrak{g}$ is a vector space $V$ together with a linear map (called the action map) $\rho : \mathfrak{g} \rightarrow \mathrm{End}(V)$ satisfying
\begin{align*}
\rho ([x,y]_\mathfrak{g}) = \rho (x) \circ \rho (y) - \rho (y) \circ \rho (x), ~ \text{ for } x, y \in \mathfrak{g}.
\end{align*}
We denote a representation as above simply by $V$ when the action map is clear from the context. Note that any Lie algebra $\mathfrak{g}$ is a representation of itself with the action map $\rho : \mathfrak{g} \rightarrow \mathrm{End}(\mathfrak{g})$ given by $\rho (x) (y) = [x, y]_\mathfrak{g}$, for $x, y \in \mathfrak{g}$. This is called the adjoint representation.

Given a Lie algebra $\mathfrak{g}$ and a representation $V$, there is a cochain complex (called the Chevalley-Eilenberg cochain complex) $\{ C^\ast_\mathsf{CE}(\mathfrak{g}, V) , \delta_\mathsf{CE} \}$, where $C^n_\mathsf{CE} (\mathfrak{g}, V) = \mathrm{Hom}(\wedge^n \mathfrak{g}, V))$ for $n \geq 0$, and the differential $\delta_\mathsf{CE} : C^n_\mathsf{CE} (\mathfrak{g}, V) \rightarrow C^{n+1}_\mathsf{CE} (\mathfrak{g}, V)$ given by
\begin{align*}
(\delta_\mathsf{CE} f)(x_1, \ldots, x_{n+1}) =~& \sum_{i=1}^{n+1} (-1)^{i+1} ~\rho (x_i) f (x_1, \ldots, \widehat{x_i}, \ldots , x_{n+1})  \\
~&+ \sum_{1 \leq i < j \leq n+1} (-1)^{i+j} ~ f ([x_i, x_j], x_1, \ldots, \widehat{x_i}, \ldots, \widehat{x_j}, \ldots, x_{n+1}),
\end{align*}
for $f \in C^n_\mathsf{CE} (\mathfrak{g}, V)$ and $x_1, \ldots, x_{n+1} \in \mathfrak{g}$. The cohomology groups of this cochain complex are called the Chevalley-Eilenberg cohomology groups of $\mathfrak{g}$ with coefficients in $V$, and they are denoted by $H^\ast_\mathsf{CE}(\mathfrak{g},V)$.

Let $\mathfrak{g}$ and $\mathfrak{h}$ be two Lie algebras. A linear map $\phi : \mathfrak{g} \rightarrow \mathfrak{h}$ is said to be a homomorphism of Lie algebras if $\phi$ satisfies $\phi ([x, y]_\mathfrak{g}) = [\phi (x), \phi (y)]_\mathfrak{h}$, for $x, y \in \mathfrak{g}$.

\begin{defn}
A morphism Lie algebra is a triple $(\mathfrak{g}, \mathfrak{h}, \phi)$ consist of two Lie algebras $\mathfrak{g}, \mathfrak{h}$ and a homomorphism $\phi : \mathfrak{g} \rightarrow \mathfrak{h}$ of Lie algebras.
\end{defn}

\begin{defn}
Let $(\mathfrak{g}, \mathfrak{h}, \phi)$ and $(\mathfrak{g}', \mathfrak{h}', \phi')$ be two morphism Lie algebras. A homomorphism between them is a pair $(\alpha, \beta)$ consists of Lie algebra homomorphisms $\alpha : \mathfrak{g} \rightarrow \mathfrak{g}'$ and $\beta : \mathfrak{h} \rightarrow \mathfrak{h}'$ satisfying $\phi' \circ \alpha = \beta \circ \phi$. It is called an isomorphism if $\alpha$ and $\beta$ are both isomorphisms.
\end{defn}

\begin{defn}
Let $(\mathfrak{g}, \mathfrak{h}, \phi)$ be a morphism Lie algebra. A representation of it consists of a triple $(V,W, \psi)$ in which $V$ is a representation of $\mathfrak{g}$, $W$ is a representation of $\mathfrak{h}$ and $\psi : V \rightarrow W$ is a linear map satisfying
\begin{align*}
\psi (\rho_V (x) v ) = \rho_W (\phi (x)) \psi (v), ~ \text{ for } x \in \mathfrak{g}, v \in V.
\end{align*}
\end{defn}

\begin{exam}
Any morphism Lie algebra $(\mathfrak{g}, \mathfrak{h}, \phi)$ is a representation of itself, where $\mathfrak{g}$ and $\mathfrak{h}$ are equipped with respective adjoint representations. We call this the adjoint representation of the morphism Lie algebra $(\mathfrak{g}, \mathfrak{h}, \phi)$.
\end{exam}

\begin{exam}\label{homo-rep}
Let $(\mathfrak{g}, \mathfrak{h}, \phi)$ and $(\mathfrak{g}', \mathfrak{h}', \phi')$ be two morphism Lie algebras, and $(\alpha, \beta)$ be a homomorphism between them. Then the triple $(\mathfrak{g}', \mathfrak{h}', \phi')$ can be regarded as a representation of the morphism Lie algebra $(\mathfrak{g}, \mathfrak{h}, \phi)$, where $\mathfrak{g}'$ is a representation of $\mathfrak{g}$ via the action map $\rho_{\mathfrak{g}'} : \mathfrak{g} \rightarrow \mathrm{End} (\mathfrak{g}')$, $\rho_{\mathfrak{g}'}(x)( x') = [\alpha (x), x']_{\mathfrak{g}'}$ and $\mathfrak{h}'$ is a representation of $\mathfrak{h}$ via the action map $\rho_{\mathfrak{h}'} : \mathfrak{h} \rightarrow \mathrm{End} (\mathfrak{h}')$, $\rho_{\mathfrak{h}'}(h)( h') = [\beta (h), h']_{\mathfrak{h}'}$.
\end{exam}

\begin{exam}
Let $\mathfrak{g}$ be a Lie algebra. A Rota-Baxter operator of weight $\lambda \in {\bf k}$ on $\mathfrak{g}$ is a linear map $R : \mathfrak{g} \rightarrow \mathfrak{g}$ satisfying
\begin{align}\label{rb-iden}
[R(x), R(y)]_\mathfrak{g} = R \big( [R(x) , y]_\mathfrak{g} + [x, R(y)]_\mathfrak{g} + \lambda [x,y]_\mathfrak{g} \big), ~\text{ for } x, y \in \mathfrak{g}.
\end{align}
Let $R$ be a Rota-Baxter operator of weight $\lambda$. Then $R$ induces a new Lie algebra structure on the underlying vector space $\mathfrak{g}$ with bracket given by
$[x,y]_R :=  [R(x) , y]_\mathfrak{g} + [x, R(y)]_\mathfrak{g} + \lambda [x,y]_\mathfrak{g}, ~\text{ for } x, y \in \mathfrak{g}.$ We denote this Lie algebra by $\mathfrak{g}_R$. It follows from (\ref{rb-iden}) that $(\mathfrak{g}_R, \mathfrak{g}, R)$ is a morphism Lie algebra.

A pair $(\mathfrak{g}, R)$ consisting of a Lie algebra $\mathfrak{g}$ and a Rota-Baxter operator of weight $\lambda$ on it is called a Rota-Baxter Lie algebra (of weight $\lambda$). A representation of a Rota-Baxter Lie algebra $(\mathfrak{g}, R)$ consists of a pair $(V, R_V)$ in which $V$ is a representation of $\mathfrak{g}$ and $R_V : V \rightarrow V$ is a linear map satisfying
\begin{align*}
\rho (R(x)) R_V(v) = R_V \big(   \rho (R(x)) v + \rho(x) R_V(v) + \lambda \rho (x) v \big), \text{ for } x \in \mathfrak{g}, v \in V.
\end{align*}
In this case, $V$ van be equipped with a representation of the Lie algebra $\mathfrak{g}_R$ with the action map $\rho_{R_V} : \mathfrak{g}_R \rightarrow \mathrm{End}(V)$ given by $\rho_{R_V}(x) v :=   \rho (R(x)) v + \rho(x) R_V(v) + \lambda \rho(x) v$, for $x \in \mathfrak{g}_R$ and $v \in V$ \cite{das-weight-lie}. We denote this representation by $V_{R_V}.$ Then it is easy to see that $(V_{R_V}, V, R_V)$ is a representation of the morphism Lie algebra $(\mathfrak{g}_R, \mathfrak{g}, R)$.
\end{exam}

\begin{exam}
Let $(\mathfrak{g}, \mathfrak{h}, \phi)$ be a morphism Lie algebra. Suppose $U$ is a representation of $\mathfrak{h}$ with the action map $\rho_U^\mathfrak{h} : \mathfrak{h} \rightarrow \mathrm{End}(U)$. Then $\phi$ induces a representation of $\mathfrak{g}$ on the vector space $U$ with action map $\rho_U^\mathfrak{g} : \mathfrak{g} \rightarrow \mathrm{End}(U),~ \rho_U^\mathfrak{g}(x) u = \rho^\mathfrak{h}_U (\phi (x)) u$, for $x \in \mathfrak{g}$, $u \in U$.

Let $V, W$ be two vector spaces and a $\psi : V \rightarrow W$ be a linear map. We define a representation of $\mathfrak{g}$ on the space $\mathrm{End}(U, V)$ by $\rho_{U,V} : \mathfrak{g} \rightarrow \mathrm{End}(\mathrm{End}(U,V))$, $(\rho_{U, V}(x) f)(u) = - f (\rho_U^\mathfrak{g} (x) u)$, for $x \in \mathfrak{g}$, $f \in \mathrm{End}(U,V)$ and $u \in U$. Similarly, there is a representation of $\mathfrak{h}$ on the space $\mathrm{End}(U, W)$ by $\rho_{U,W} : \mathfrak{h} \rightarrow \mathrm{End}(\mathrm{End}(U,W))$, $(\rho_{U,W} (h) g)(u) = - g (\rho_U^\mathfrak{h} (h) u)$, for $h \in \mathfrak{h}$, $g \in \mathrm{End}(U,W)$ and $u \in U$.  We also define a map $\Psi : \mathrm{End}(U,V) \rightarrow \mathrm{End}(U,W)$ by $\Psi (f) = \psi \circ f$. For $x \in \mathfrak{g}$, $f \in \mathrm{End}(U, V)$ and $u \in U$, we observe that
\begin{align*}
(\Psi (\rho_{U,V} (x) f )) u = \psi (  (\rho_{U,V} (x) f ) u ) = - (\psi \circ f) (\rho_U^\mathfrak{g} (x) u) = - \Psi (f) ( \rho_U^\mathfrak{h} (\phi (x)) u) = (\rho_{U,W} (\phi (x)) \Psi (f)) u.
\end{align*}
This shows that $(\mathrm{End}(U,V), \mathrm{End}(U,W), \Psi)$ is a representation of the morphism Lie algebra $(\mathfrak{g}, \mathfrak{h}, \phi)$.
\end{exam}

\medskip

It is known that a representation of a Lie algebra $\mathfrak{g}$ can be equivalently described as a left module over the universal enveloping algebra $U(\mathfrak{g})$. Note that $U(\mathfrak{g})$ is obtained from the tensor algebra $T(\mathfrak{g})$ quotient by the two-sided ideal $I_\mathfrak{g}$ generated by elements of the form $x \otimes y - y \otimes x - [x,y]_\mathfrak{g}$, for $x, y \in \mathfrak{g}$. More precisely, if $V$ is a left $U(\mathfrak{g})$-module, then the representation of the Lie algebra $\mathfrak{g}$ on $V$ is given by the action map $\rho_V : \mathfrak{g} \rightarrow \mathrm{End}(V)$, $\rho_V (x) v = x \cdot v$, for $x \in \mathfrak{g}, v \in V$. We will generalize this result to representations of morphism Lie algebras.

Let $(\mathfrak{g}, \mathfrak{h}, \phi)$ be a morphism Lie algebra. Then $\phi$ induces an algebra homomorphism $T(\phi) : T (\mathfrak{g}) \rightarrow T (\mathfrak{h})$ between tensor algebras given by
\begin{align*}
T(\phi) (x_1 \otimes \cdots \otimes x_n) = \phi (x_1) \otimes \cdots \otimes \phi (x_n),~ \text{ for } x_1 \otimes \cdots \otimes x_n \in \mathfrak{g}^{\otimes n} \subset T(\mathfrak{g}).
\end{align*}
By composing with the projection $T (\mathfrak{h}) \rightarrow U(\mathfrak{h}) = T(\mathfrak{h})/ I_\mathfrak{h}$, we obtain a map (denoted by the same notation) $T(\phi) : T (\mathfrak{g}) \rightarrow U (\mathfrak{h})$. Since
\begin{align*}
T(\phi) (x \otimes y - y \otimes x - [x, y]_\mathfrak{g}) = \phi(x) \otimes \phi(y) - \phi(y) \otimes \phi(x) - [\phi(x), \phi(y)]_\mathfrak{h} = 0 ~~~(\text{mod } I_\mathfrak{h}),
\end{align*}
the map $T(\phi)$ induces an algebra homomorphism $U(\phi) : U(\mathfrak{g}) \rightarrow U(\mathfrak{h})$. In other words, the triple $(U(\mathfrak{g}), U(\mathfrak{h}), U(\phi))$ is a morphism associative algebra.

A left module over the associative algebra $(U(\mathfrak{g}), U(\mathfrak{h}), U(\phi))$ is a triple $(V, W, \psi)$ in which $V$ is a left $U(\mathfrak{g})$-module, $W$ is a left $U(\mathfrak{h})$-module and $\phi : V \rightarrow W$ is a linear map satisfying $\psi (x \cdot v) = U(\phi)(x) \cdot \psi (v)$, for $x \in U(\mathfrak{g}), v \in V$. Then we have the following result.

\begin{prop}
Let $(\mathfrak{g}, \mathfrak{h}, \phi)$ be a morphism Lie algebra. A representation of $(\mathfrak{g}, \mathfrak{h}, \phi)$ is equivalent to a left module over the morphism associative algebra $(U(\mathfrak{g}), U(\mathfrak{h}), U(\phi))$.
\end{prop}

\section{Cohomology of morphism Lie algebras}\label{section-cohomology}

In this section, we introduce the cohomology of a morphism Lie algebra with coefficients in a representation. Our cohomology generalizes the cohomology of Lie algebra homomorphisms defined in \cite{fregier}.

Let $(\mathfrak{g}, \mathfrak{h}, \phi)$ be a morphism Lie algebra and $(V,W,\psi)$ be a representation of it. Then there are two obvious cochain complexes, namely
\begin{itemize}
\item $\{ C^\ast_\mathsf{CE}(\mathfrak{g}, V), \delta'_\mathsf{CE} \}$, the Chevalley-Eilenberg complex of $\mathfrak{g}$ with coefficients in $V$,
\item $\{ C^\ast_\mathsf{CE}(\mathfrak{h}, W), \delta''_\mathsf{CE} \}$, the Chevalley-Eilenberg complex of $\mathfrak{h}$ with coefficients in $W$.
\end{itemize}

\medskip

We define a new map $\rho_W^\mathfrak{g} : \mathfrak{g} \rightarrow \mathrm{End}(W)$ by $\rho_W^\mathfrak{g} (x) w = \rho_W (\phi (x)) w$, for $x \in \mathfrak{g}$ and $w \in W$. For any $x, y \in \mathfrak{g}$ and $w \in W$, we have
\begin{align*}
\rho_W^\mathfrak{g} ([x, y]_\mathfrak{g}) w = \rho_W ([\phi (x), \phi (y)]_\mathfrak{h}) w =~& \rho_W (\phi (x)) \rho_W (\phi (y)) w ~-~ \rho_W (\phi (y)) \rho_W (\phi (x)) w \\
=~& \rho_W^\mathfrak{g} (x) \rho_W^\mathfrak{g}(y) w ~-~ \rho_W^\mathfrak{g} (y) \rho_W^\mathfrak{g}(x) w.
\end{align*}
This shows that $\rho_W^\mathfrak{g}$ defines a representation of the Lie algebra $\mathfrak{g}$ on the vector space $W$. We denote this representation by $W_\phi$. Hence we may consider another cochain complex, namely
\begin{itemize}
\item $\{ C^\ast_\mathsf{CE}(\mathfrak{g}, W_\phi), \delta'''_\mathsf{CE} \}$, the Chevalley-Eilenberg complex of $\mathfrak{g}$ with coefficients in $W_\phi$.
\end{itemize}

We are now in a position to define the cohomology of the morphism Lie algebra $(\mathfrak{g}, \mathfrak{h}, \phi)$ with coefficients in $(V,W, \psi)$. For each $n \geq 0$, we define the $n$-th cochain group $C^n_\mathsf{mLA} (\phi, \psi)$ by
\begin{align*}
C^n_\mathsf{mLA} (\phi, \psi) = \begin{cases}
V & \quad \text{ if } n = 0,\\
C^n_\mathsf{CE}(\mathfrak{g}, V) \oplus C^n_\mathsf{CE}(\mathfrak{h}, W) \oplus C^{n-1}_\mathsf{CE}(\mathfrak{g}, W_\phi) \\
 =\mathrm{Hom} (\wedge^n \mathfrak{g}, V) \oplus \mathrm{Hom}(\wedge^n \mathfrak{h}, W) \oplus \mathrm{Hom}(\wedge^{n-1} \mathfrak{g}, W)   & \quad \text{ if } n \geq 1
 \end{cases}
\end{align*}
and a map $\delta_\mathsf{mLA} : C^n_\mathsf{mLA} (\phi, \psi) \rightarrow C^{n+1}_\mathsf{mLA} (\phi, \psi)$ by
\begin{align*}
\delta_\mathsf{mLA} (v) =~& (\delta'_\mathsf{CE} (v),~ \delta''_\mathsf{CE} (\psi (v)),~ 0), ~ \text{ for } v \in C^0_\mathsf{mLA}(\phi, \psi), \\
\delta_\mathsf{mLA} (\theta, \gamma, \eta) =~& (\delta'_\mathsf{CE}(\theta),~ \delta''_\mathsf{CE}(\gamma),~ \psi \circ \theta - \gamma \circ {\wedge^n \phi} - \delta'''_\mathsf{CE}(\eta)),~ \text{ for } (\theta, \gamma, \eta) \in C^{n \geq 1}_\mathsf{mLA}(\phi, \psi).
\end{align*}

\begin{prop}
With the above notations, we have $(\delta_\mathsf{mLA})^2=0.$
\end{prop}

\begin{proof}
For $v \in C^0_\mathsf{mLA}(\phi, \psi)$, 
\begin{align*}
(\delta_\mathsf{mLA})^2 (v) =~& \delta_\mathsf{mLA} \big(  \delta'_\mathsf{CE} (v), ~ \delta''_\mathsf{CE} (\psi(v)),~0 \big) \\
=~& \big(  (\delta'_\mathsf{CE})^2 (v), ~ (\delta''_\mathsf{CE})^2 (\psi(v)),~ \psi \circ \delta'_\mathsf{CE} (v) - \delta''_\mathsf{CE} (\psi (v)) \circ \phi \big) = 0.
\end{align*}
The lase equality holds as $(\psi \circ \delta'_\mathsf{CE} (v)) (x) = \psi (\rho_V(x) v) = \rho_W (\phi (x))\psi (v) = (\delta''_\mathsf{CE} (\psi (v)) \circ \phi) (x)$, for all $x \in \mathfrak{g}$. Similarly, for $(\theta, \gamma, \eta) \in C^{n \geq 1}_\mathsf{mLA} (\phi, \psi),$ we have
\begin{align}\label{d2-iden}
&(\delta_\mathsf{mLA})^2 (\theta, \gamma, \eta) \nonumber \\
&= \delta_\mathsf{mLA} (\delta'_\mathsf{CE}(\theta),~ \delta''_\mathsf{CE}(\gamma),~ \psi \circ \theta - \gamma \circ {\wedge^n \phi} - \delta'''_\mathsf{CE}(\eta)) \nonumber \\
&= ( (\delta'_\mathsf{CE})^2 (\theta),~(\delta''_\mathsf{CE})^2 (\gamma),~\psi \circ \delta'_\mathsf{CE} (\theta) - \delta''_\mathsf{CE} (\gamma) \circ \wedge^{n+1} \phi - \delta'''_\mathsf{CE} (\psi \circ \theta) + \delta'''_\mathsf{CE} (\gamma \circ \wedge^n \phi) + (\delta'''_\mathsf{CE})^2 (\eta) ).
\end{align}
We observe that
\begin{align*}
\delta'''_\mathsf{CE} (\psi \circ \theta) (x_1, \ldots, x_{n+1}) 
=~& \sum_{i=1}^{n+1} (-1)^{i+1} ~\rho_W^\mathfrak{g} (x_i) (\psi \circ \theta)(x_1, \ldots, \widehat{x_i}, \ldots, x_{n+1}) \\
~& + \sum_{1 \leq i < j \leq n+1} (-1)^{i+j} ~ (\psi \circ \theta) ([x_i, x_j]_\mathfrak{g}, x_1, \ldots, \widehat{x_i}, \ldots, \widehat{x_j}, \ldots, x_{n+1} ) \\
=~& \psi \big(  \sum_{i=1}^{n+1} (-1)^{i+1} ~\rho_V (x_i)  \theta (x_1, \ldots, \widehat{x_i}, \ldots, x_{n+1}) \\
~& + \sum_{1 \leq i < j \leq n+1} (-1)^{i+j} ~ (\psi \circ \theta) ([x_i, x_j]_\mathfrak{g}, x_1, \ldots, \widehat{x_i}, \ldots, \widehat{x_j}, \ldots, x_{n+1} ) \big) \\
=~& (\psi \circ \delta'_\mathsf{CE} (\theta)) (x_1, \ldots, x_{n+1})
\end{align*}
and
\begin{align*}
\delta'''_\mathsf{CE} (\gamma \circ \wedge^n \phi) (x_1, \ldots, x_{n+1}) = ~& \sum_{i=1}^{n+1} (-1)^{i+1} ~\rho_W^\mathfrak{g} (x_i) (\gamma \circ \wedge^n \phi)(x_1, \ldots, \widehat{x_i}, \ldots, x_{n+1}) \\
~& + \sum_{1 \leq i < j \leq n+1} (-1)^{i+j} ~ (\gamma \circ \wedge^n \phi) ([x_i, x_j]_\mathfrak{g}, x_1, \ldots, \widehat{x_i}, \ldots, \widehat{x_j}, \ldots, x_{n+1} ) \\
=~& \sum_{i=1}^{n+1} (-1)^{i+1}~ \rho_W (\phi (x_i)) \gamma (\phi (x_1), \ldots, \widehat{\phi (x_i)}, \ldots, \phi (x_{n+1})) \\
~& + \sum_{1 \leq i < j \leq n+1} (-1)^{i+j} ~\gamma ([\phi(x_i), \phi(x_j)]_\mathfrak{h}, \phi (x_1), \ldots, \widehat{\phi(x_i)}, \ldots, \widehat{\phi (x_j)}, \ldots )\\
=~& (\delta''_\mathsf{CE} (\gamma) \circ \wedge^{n+1} \phi) (x_1, \ldots, x_{n+1}).
\end{align*}
Hence it follows from (\ref{d2-iden}) that $(\delta_\mathsf{mLA})^2 = 0$. This completes the proof.
\end{proof}

It follows from the above proposition that $\{ C^\ast_\mathsf{mLA}(\phi, \psi), \delta_\mathsf{mLA} \}$ is a cochain complex. Let $Z^n_\mathsf{mLA}(\phi, \psi)$ denote the space of $n$-cocycles and $B^n_\mathsf{mLA}(\phi, \psi)$  denote the space of $n$-coboundaries. Then we have $B^n_\mathsf{mLA}(\phi, \psi) \subset Z^n_\mathsf{mLA}(\phi, \psi)$. The corresponding quotients
\begin{align}\label{mla-co}
H^n_\mathsf{mLA}(\phi, \psi) := \frac{Z^n_\mathsf{mLA}(\phi, \psi)}{ B^n_\mathsf{mLA}(\phi, \psi)}, ~\text{ for } n \geq 0
\end{align}
are called the cohomology groups of the morphism Lie algebra $(\mathfrak{g}, \mathfrak{h}, \phi)$ with coefficients in the representation $(V,W,\psi)$.


It follows from the above definition that
\begin{align*}
H^0_\mathsf{mLA} (\phi, \psi) = \{ v \in V ~|~ \rho_V(x) v = 0 \text{ and } \rho_W (h) \psi (v) = 0, \text{ for all } x \in \mathfrak{g}, h \in \mathfrak{h} \}.
\end{align*}
A triple $(d, \delta, w) \in \mathrm{Hom}(\mathfrak{g}, V) \oplus \mathrm{Hom}(\mathfrak{h}, W) \oplus W$ is said to be a derivation for the morphism Lie algebra $(\mathfrak{g}, \mathfrak{h}, \phi)$ with coefficients in $(V,W,\psi)$ if
\begin{align*}
d ([x, y]_\mathfrak{g}) =~& \rho_V (x) d (y) - \rho_V(y) d (x), ~\text{ for } x, y \in \mathfrak{g}, \\
\delta ([h, k]_\mathfrak{h}) =~& \rho_W (h) \delta (k) - \rho_W (k) \delta (h), ~ \text{ for } h, k \in \mathfrak{h},\\
\rho_W (\phi (x)) w =~& \psi (d(x)) - \delta (\phi (x)), \text{ for } x \in \mathfrak{g}.
\end{align*}
In other words, $d$ and $\delta$ are both derivations on Lie algebras $\mathfrak{g}$ and $\mathfrak{h}$ respectively, and the presence of $w$ obstructs the triviality of $\psi \circ d - \delta \circ \phi$. We denote the set of all derivations by $\mathrm{Der}(\phi, \psi)$. An inner derivation is a derivation of the form $(\delta'_\mathsf{CE} (v),~ \delta''_\mathsf{CE} (\psi (v)),~0)$, for some $v \in V$. The set of all inner derivations are denoted by $\mathrm{InnDer}(\phi, \psi)$. It follows from (\ref{mla-co}) that $H^1_\mathsf{mLA}(\phi, \psi) = \frac{ \mathrm{Der}(\phi, \psi) }{\mathrm{InnDer}(\phi, \psi)}$, called the space of outer derivations.

\begin{remark}
Let $(\mathfrak{g}, \mathfrak{h}, \phi)$ be a morphism Lie algebra. In \cite{fregier} Fr\'{e}gier introduced a cohomology that governs deformation of the morphism Lie algebra $(\mathfrak{g}, \mathfrak{h}, \phi)$. The cohomology of \cite{fregier} can be seen as our cohomology of the morphism Lie algebra $(\mathfrak{g}, \mathfrak{h}, \phi)$ with coefficients in the adjoint representation.
\end{remark}

In the following, we give a sufficient condition for the vanishing of the cohomology groups $H^\ast_\mathsf{mLA}(\phi, \psi)$ in terms of the vanishing of some Chevalley-Eilenberg cohomology groups.

\begin{prop}\label{vanish-prop}
Let $(\mathfrak{g}, \mathfrak{h}, \phi)$ be a morphism Lie algebra and $(V,W, \psi)$ be a representation of it. If $H^n_\mathsf{CE} (\mathfrak{g}, V)$, $H^n_\mathsf{CE} (\mathfrak{h}, W)$ and $H^{n-1}_\mathsf{CE} (\mathfrak{g}, W_\phi)$ are trivial, so is $H^n_\mathsf{mLA} (\phi, \psi)$.
\end{prop}

\begin{proof}
Let $(\theta, \gamma, \eta) \in Z^n_\mathsf{mLA}(\phi, \psi)$ be an $n$-cocycle. It follows that $\theta \in Z^n_\mathsf{CE}(\mathfrak{g}, V)$ and $\gamma \in Z^n_\mathsf{CE} (\mathfrak{h}, W)$ are $n$-cocycles in the respective Chevalley-Eilenberg cochain complexes and $\psi \circ \theta - \gamma \circ \wedge^n \phi - \delta'''_\mathsf{CE}(\eta) = 0$. Hence by the hypothesis, there exist $(n-1)$-cochains $\theta_0 \in C^{n-1}_\mathsf{CE}(\mathfrak{g}, V)$ and $\gamma_0 \in C^{n-1}_\mathsf{CE}(\mathfrak{h}, W)$ such that $\theta = \delta'_\mathsf{CE}(\theta_0)$ and $\gamma = \delta''_\mathsf{CE}({\gamma}_0)$. This implies that
\begin{align*}
\delta'''_\mathsf{CE} (\psi \circ \theta_0 - \gamma_0 \circ \wedge^{n-1} \phi - \eta) =~& \psi \circ \delta'_\mathsf{CE} (\theta_0) - \delta''_\mathsf{CE} (\gamma_0) \circ \wedge^n \phi - \delta'''_\mathsf{CE} (\eta) \\
=~& \phi \circ \theta - \gamma \circ \wedge^n \phi - \delta'''_\mathsf{CE} (\eta) = 0.
\end{align*}
In other words, $(\psi \circ \theta_0 - \gamma_0 \circ \wedge^{n-1} \phi - \eta) \in Z^{n-1}_\mathsf{CE} (\mathfrak{g}, W_\phi)$ is a $(n-1)$-cocycle. Hence by the hypothesis, there exists an element $\eta_0 \in C^{n-1}_\mathsf{CE}(\mathfrak{g}, W_\phi)$ such that $(\psi \circ \theta_0 - \gamma_0 \circ \wedge^{n-1} \phi - \eta) = \delta'''_\mathsf{CE}(\eta_0).$ This implies that $(\theta, \gamma, \eta) = \delta_\mathsf{mLA} (\theta_0, \gamma_0, \eta_0) \in B^n_\mathsf{mLA}(\phi, \psi)$ is a $n$-coboundary. Hence the result follows.
\end{proof}

\begin{corollary}
Let $(\mathfrak{g}, \mathfrak{h}, \phi)$ be a morphism Lie algebra and $(V,W, \psi)$ be a representation of it. Suppose $\mathfrak{g}, \mathfrak{h}$ are both semisimple Lie algebras and $V, W$ are nontrivial irreducible representations of $\mathfrak{g}$ and $\mathfrak{h}$, respectively. Additionally, we assume that the representation $W_\phi$ of the Lie algebra $\mathfrak{g}$ is nontrivial and irreducible. Then by the Whitehead's theorem \cite{weibel} and Proposition \ref{vanish-prop}, we get that $H^\ast_\mathsf{mLA} (\phi, \psi) = 0$.
\end{corollary}

\section{Deformation theories}\label{sec-4}
Deformations of morphism Lie algebras was considered by Fr\'{e}gier in \cite{fregier}. In a deformation of a morphism Lie algebra $(\mathfrak{g}, \mathfrak{h}, \phi)$, we simultaneously deform the underlying Lie algebras $\mathfrak{g}, \mathfrak{h}$ and the Lie algebra homomorphism $\phi$. The main result of \cite{fregier} says that the cohomology (with coefficients in the adjoint representation) of the morphism Lie algebra $(\mathfrak{g}, \mathfrak{h}, \phi)$ govern such deformations. In the following, we study some other aspects of deformation problems in the context of morphism Lie algebras.


\subsection*{Deformations of homomorphisms between morphism Lie algebras} 

Let $(\mathfrak{g}, \mathfrak{h}, \phi)$ and $(\mathfrak{g}', \mathfrak{h}', \phi')$ be two morphism Lie algebras and $(\alpha, \beta)$ be a homomorphism between them. We have seen in Example \ref{homo-rep} that the triple $(\mathfrak{g}', \mathfrak{h}', \phi')$  can be considered as a representation of the morphism Lie algebra $(\mathfrak{g}, \mathfrak{h}, \phi)$, where the representation of the Lie algebra $\mathfrak{g}$ on the vector space $\mathfrak{g}'$ is given by $\rho_{\mathfrak{g}'} : \mathfrak{g} \rightarrow \mathrm{End}(\mathfrak{g}'), ~\rho_{\mathfrak{g}'} (x) (x') = [\alpha(x), x']_{\mathfrak{g}'}$, and the representation of $\mathfrak{h}$ on the space $\mathfrak{h}'$ is given by $\rho_{\mathfrak{h}'} : \mathfrak{h} \rightarrow \mathrm{End}(\mathfrak{h}'),~ \rho_{\mathfrak{h}'} (h) (h') = [\beta (h), h']_{\mathfrak{h}'}$. Hence we may consider the cochain complex $\{ C^\ast_\mathsf{mLA} (\phi, \phi'), \delta_\mathsf{mLA} \}$, where
\begin{align*}
C^0_\mathsf{mLA} (\phi, \phi') = \mathfrak{g}' \quad \text{ and } \quad C^{n \geq 1}_\mathsf{mLA} (\phi, \phi') = \mathrm{Hom} (\wedge^n \mathfrak{g}, \mathfrak{g}') \oplus  \mathrm{Hom} (\wedge^n \mathfrak{h}, \mathfrak{h}') \oplus  \mathrm{Hom} (\wedge^{n-1} \mathfrak{g}, \mathfrak{h}').
\end{align*}
We will use the corresponding cohomology $H^\ast_\mathrm{mLA} (\phi, \phi')$ to study deformations of the homomorphism $(\alpha, \beta)$.

\begin{defn}
A deformation of $(\alpha, \beta)$ consists of a pair $(\alpha_t, \beta_t)$ of two formal power series
\begin{align*}
\alpha_t = \sum_{i \geq 0} \alpha_i t^i ~\text{ and } ~ \beta_t = \sum_{i \geq 0} \beta_i t^i ~~~ (\text{with }\alpha_i \in \mathrm{Hom}(\mathfrak{g}, \mathfrak{g}'),~ \beta_i \in \mathrm{Hom}(\mathfrak{h}, \mathfrak{h}') \text{ and } \alpha_0 = \alpha,~ \beta_0 = \beta)
\end{align*}
satisfying for $x,y \in \mathfrak{g}$ and $h, k \in \mathfrak{h},$
\begin{align*}
\alpha_t ([x, y]_\mathfrak{g}) = [\alpha_t (x), \alpha_t (y)]_{\mathfrak{g}'}, \quad \beta_t ([h,k]_\mathfrak{h}) = [\beta_t (h), \beta_t (k)]_{\mathfrak{h}'} ~~~ \text{ and } ~~~ \phi' \circ \alpha_t = \beta_t \circ \phi.
\end{align*}
\end{defn}

By equating coefficients of $t$ in each of the above identities, we get
\begin{align*}
\alpha_1 ([x, y]_\mathfrak{g}) =~& [ \alpha (x), \alpha_1 (y)]_{\mathfrak{g}'} + [ \alpha_1 (x), \alpha (y)]_{\mathfrak{g}'},\\
\beta_1 ([h,k]_\mathfrak{h}) =~& [\beta_1 (h), \beta_1 (k)]_{\mathfrak{h}'} + [\beta_1 (h), \beta (k)]_{\mathfrak{h}'},\\
\phi' \circ \alpha_1 =~& \beta_1 \circ \phi.
\end{align*}
The first condition is equivalent to $\delta'_\mathsf{CE} (\alpha_1) (x,y) = 0$, for $x,y \in \mathfrak{g}$, and the second condition is equivalent to $\delta''_\mathsf{CE} (\beta_1) (h, k) = 0$. Therefore, we have
\begin{align*}
\delta_\mathsf{mLA}(\alpha_1, \beta_1, 0) = (  \delta'_\mathsf{CE} (\alpha_1),~\delta''_\mathsf{CE} (\beta_1),~ \phi' \circ \alpha_1 - \beta_1 \circ \phi  ) = 0.
\end{align*}
This shows that $(\alpha_1, \beta_1, 0) \in C^1_\mathsf{mLA} (\phi, \phi')$ is a $1$-cocycle (that is, lies in $Z^1_\mathsf{mLA} (\phi, \phi')$). This is called the infinitesimal (or linear component) of the deformation. In particular, if $(\alpha_1, \beta_1, 0) = \cdots = (\alpha_{n-1}, \beta_{n-1}, 0) = 0$ and $(\alpha_n, \beta_n, 0)$ is nonzero, then $(\alpha_n, \beta_n, 0) \in Z^1_\mathsf{mLA} (\phi, \phi')$ is a $1$-cocycle, called the $n$-th infinitesimal.

\medskip

Next we define an equivalence relation on the space of space of all deformations. Let $G'$ and $H'$ be the unique simply connected Lie groups integrating the Lie algebras $\mathfrak{g}'$ and $\mathfrak{h}'$, respectively. Then the Lie algebra homomorphism $\phi' : \mathfrak{g}' \rightarrow \mathfrak{h}'$ integrates to a Lie group homomorphism $\Phi' : G' \rightarrow H'$.

\begin{defn}
Two deformations $(\alpha_t, \beta_t)$ and $(\alpha'_t, \beta'_t)$ are said to be equivalent if there exists a smooth curve $g_t'$ in $G'$ starting at the identity (i.e. $g_0' = 1_{G'}$) such that
\begin{align}\label{mor-equi-ad}
\alpha_t' (x) = \mathrm{Ad}^{G'}_{g'_t} \circ \alpha_t (x) ~~~ \text{ and } ~~~ \beta_t' (h) = \mathrm{Ad}^{H'}_{\Phi' (g_t')} \circ \beta_t (h), ~ \text{ for } x \in \mathfrak{g}, h \in \mathfrak{h}.
\end{align}
\end{defn}

By differentiating both the identities of (\ref{mor-equi-ad}) at $t=0$, we get
\begin{align*}
(\alpha_1 - \alpha_1')(x) = \delta'_\mathsf{CE} (\frac{d}{dt}|_{t=0}~ g_t')(x)   ~~~ \text{ and } ~~~  (\beta_1 - \beta_1')(h) = \delta''_\mathsf{CE} (\frac{d}{dt}|_{t=0}~ \Phi'(g_t') )(h).
\end{align*}
In other words, $(\alpha_1, \beta_1, 0) - (\alpha_1', \beta_1', 0) = \delta_\mathsf{mLA} (  \frac{d}{dt}|_{t=0}~ g_t'  )$, where the derivative $\frac{d}{dt}|_{t=0}~ g_t' \in \mathfrak{g}' = C^0_\mathsf{mLA}(\phi, \phi')$. As a summary of the above discussions, we get the following.

\begin{prop}\label{prop-def-mor}
Let $(\alpha_t, \beta_t)$ be a deformation of the homomorphism $(\alpha, \beta)$ of morphism Lie algebras. Then the infinitesimal is $1$-cocycle in $Z^1_\mathsf{mLA} (\phi, \phi')$. Moreover, the corresponding cohomology class in $H^1_\mathsf{mLA} (\phi, \phi')$ depends only on the equivalence class of the deformation $(\alpha_t, \beta_t)$.
\end{prop}

\medskip

Next, we consider the rigidity of a homomorphism between morphism Lie algebras and find a sufficient condition for the rigidity.

\begin{defn}
A homomorphism $(\alpha, \beta)$ is said to be rigid if any deformation $(\alpha_t, \beta_t)$ of $(\alpha, \beta)$ is equivalent to the undeformed one $(\overline{\alpha}_t = \alpha, \overline{\beta}_t = \beta)$.
\end{defn}

\begin{thm}
Let $(\alpha, \beta)$ be a homomorphism of morphism Lie algebras from $(\mathfrak{g}, \mathfrak{h}, \phi)$ to  $(\mathfrak{g}', \mathfrak{h}', \phi')$. If $H^1_\mathsf{mLA} (\phi, \phi') = 0$ then the homomorphism $(\alpha, \beta)$ is rigid.
\end{thm}

\begin{proof}
Let $(\alpha_t, \beta_t)$ be any deformation of $(\alpha, \beta)$. Then we know from Proposition \ref{prop-def-mor} that $(\alpha_1, \beta_1, 0) \in Z^1_\mathsf{mLA} (\phi, \phi')$ is a $1$-cocycle. Thus, from the hypothesis, there exists an element $x' \in \mathfrak{g}'$ such that $\alpha_1 = \delta'_\mathsf{CE}(x')$ and $\beta_1 = \delta''_\mathsf{CE} (\phi' (x'))$. Let $g'_t$ be a smooth curve in $G'$ starting at the identity $1_{G'}$ such that $\frac{d}{dt} |_{t=0} g'_t = x'$ and $\frac{d}{dt}|_{t=0} \Phi' (g_t') = \phi' (x')$. We define a deformation $(\alpha_t', \beta_t')$ by $\alpha_t' = \mathrm{Ad}^{G'}_{g_t'} \circ \alpha_t$ and $\beta_t' = \mathrm{Ad}^{H'}_{\Phi'(g_t')} \circ \beta_t$. Then $(\alpha_t', \beta'_t)$ is equivalent to $(\alpha_t, \beta_t)$. Moreover, we have
\begin{align*}
\alpha_1' = \frac{d}{dt}|_{t= 0} \alpha_t' = 0 ~~~~ \text{ and } ~~~~ \beta_1' = \frac{d}{dt}|_{t= 0} \beta_t' = 0.
\end{align*}
Therefore, the linear component $(\alpha_1', \beta_1', 0)$ associated to the deformation $(\alpha_t', \beta_t')$ is null. By repeating this process, we can show that $(\alpha_t, \beta_t)$ is equivalent to $(\overline{\alpha}_t = \alpha, \overline{\beta}_t = \beta)$. This completes the proof.
\end{proof}
\subsection*{Deformations of morphism Lie subalgebras}

Here we study infinitesimal deformation theory of morphism Lie subalgebras in terms of cohomology.

\begin{defn}
Let $(\mathfrak{g}, \mathfrak{h}, \phi)$ be a morphism Lie algebra. A morphism Lie subalgebra of $(\mathfrak{g}, \mathfrak{h}, \phi)$ is a triple $(\mathfrak{p}, \mathfrak{q}, \varphi)$ in which $\mathfrak{p} \hookrightarrow \mathfrak{g}$ and $\mathfrak{q} \hookrightarrow \mathfrak{h}$ are Lie subalgebras and $\varphi : \mathfrak{p} \rightarrow \mathfrak{q}$ is a linear map satisfying $\varphi = \phi|_\mathfrak{p}$. We often write $(\mathfrak{p}, \mathfrak{q}, \varphi) \hookrightarrow (\mathfrak{g}, \mathfrak{h}, \phi)$ to denote a morphism Lie subalgebra.
\end{defn}

It follows that $(\mathfrak{p}, \mathfrak{q}, \varphi)$ is a morphism Lie algebra in its own right and the pair of inclusion maps defines a homomorphism of morphism Lie algebras from $(\mathfrak{p}, \mathfrak{q}, \varphi)$ to $(\mathfrak{g}, \mathfrak{h}, \phi)$.

Let $(\mathfrak{p}, \mathfrak{q}, \varphi) \hookrightarrow (\mathfrak{g}, \mathfrak{h}, \phi)$ be a morphism Lie subalgebra. Note that the composition $\mathfrak{g} \xrightarrow{\phi} \mathfrak{h} \rightarrow \mathfrak{h}/ \mathfrak{q}$ induces a map (denoted by $\phi/ \varphi$) $\mathfrak{g}/ \mathfrak{p} \rightarrow \mathfrak{h} / \mathfrak{q}$. Moreover, there is a representation of the Lie algebra $\mathfrak{p}$ on the space $\mathfrak{g} / \mathfrak{p}$ given by the action map $\rho_{ \mathfrak{g} / \mathfrak{p} } (p) (x \text{ mod } \mathfrak{p}) = [p, x]_\mathfrak{g} \text{ mod } \mathfrak{p}$, for $p \in \mathfrak{p}$ and $x \in \mathfrak{g}$. Similarly, there is a representation of the Lie algebra $\mathfrak{q}$ on the space $\mathfrak{h} / \mathfrak{q}$ with the action map $\rho_{ \mathfrak{h} / \mathfrak{q} } (q) (h \text{ mod } \mathfrak{q}) = [q, h]_\mathfrak{h} \text{ mod } \mathfrak{q}$, for $q \in \mathfrak{q}$ and $h \in \mathfrak{h}$. It is further easy to see that the triple $(\mathfrak{g}/ \mathfrak{p},~ \mathfrak{h}/ \mathfrak{q},~ \phi / \varphi)$ is a representation of the morphism Lie algebra $(\mathfrak{p}, \mathfrak{q}, \varphi)$. Therefore, we may consider the cochain complex $\{ C^\ast_\mathsf{mLA} (\varphi, \phi / \varphi), \delta_\mathsf{mLA} \}$ and the corresponding cohomology groups $H^\ast_\mathsf{mLA} (\varphi, \phi / \varphi)$.

Let $\mathrm{dim }(\mathfrak{p}) = d$ and $\mathrm{dim }(\mathfrak{q}) = l$. We consider the following Grassmannians of $d$-dimensional subspaces of $\mathfrak{g}$ and $l$-dimensional subspaces of $\mathfrak{h}$, i.e.
\begin{align*}
\mathrm{Gr}_d (\mathfrak{g}) =~& \{ P \subset \mathfrak{g} ~|~ P \text{ is a } d\text{-dimensional subspace of } \mathfrak{g} \},\\
\mathrm{Gr}_l (\mathfrak{h}) =~& \{ Q \subset \mathfrak{h} ~|~ Q \text{ is a } l\text{-dimensional subspace of } \mathfrak{h} \}.
\end{align*}

\noindent {\bf Observation.} \cite[Remark 3.13]{crainic} Note that $\mathfrak{p} \in \mathrm{Gr}_d (\mathfrak{g})$. Moreover, the tangent space $T_\mathfrak{p} ( \mathrm{Gr}_d (\mathfrak{g}))$ can be canonically identified with $\mathfrak{p}^* \otimes \mathfrak{g} / \mathfrak{p} = \mathrm{Hom} (\mathfrak{p}, \mathfrak{g}/ \mathfrak{p})$ as follows. If $\mathfrak{p}_t$ is a smooth curve in  $\mathrm{Gr}_d (\mathfrak{g})$ starting at $\mathfrak{p}$, we can find a curve $a_t$ in $\mathrm{GL}(\mathfrak{g})$ starting at the identity and such that $\mathfrak{p}_t = a_t (\mathfrak{p})$. Then the tangent vector $\dot {\mathfrak{p}}_0 = \frac{d}{dt}|_{t=0} \mathfrak{p}_t \in T_\mathfrak{p}  (\mathrm{Gr}_d (\mathfrak{g}))$ can be considered as an element in $\mathfrak{p}^* \otimes \mathfrak{g}/ \mathfrak{p} = \mathrm{Hom}(\mathfrak{p}, \mathfrak{g}/ \mathfrak{p})$ by
\begin{align}\label{p-der}
\dot {\mathfrak{p}}_0 (p) = \frac{d}{dt}|_{t=0} a_t (p) \text{ mod } \mathfrak{p}, ~\text{ for } p \in \mathfrak{p}.
\end{align}
Similarly, the tangent space $T_\mathfrak{q} ( \mathrm{Gr}_l (\mathfrak{h}))$ can be identified with the space $\mathfrak{q}^* \otimes \mathfrak{h} / \mathfrak{q} = \mathrm{Hom} (\mathfrak{q}, \mathfrak{h}/ \mathfrak{q}).$

Let $(\mathfrak{g}, \mathfrak{h}, \phi)$ be a morphism Lie algebra and $(\mathfrak{p}, \mathfrak{q}, \varphi) \hookrightarrow (\mathfrak{g}, \mathfrak{h}, \phi)$ be a morphism Lie subalgebra. Since the map $\varphi$ is the restriction of $\phi$ to the subspace $\mathfrak{p}$, 
the above morphism Lie subalgebra can be written as $(\mathfrak{p}, \mathfrak{q}, \phi|_{\mathfrak{p}})$.


\begin{defn}
Let $(\mathfrak{g}, \mathfrak{h}, \phi)$ be a morphism Lie algebra and $(\mathfrak{p}, \mathfrak{q}, \phi|_\mathfrak{p}) \hookrightarrow (\mathfrak{g}, \mathfrak{h}, \phi)$ be a morphism Lie subalgebra. A deformation of $(\mathfrak{p}, \mathfrak{q}, \phi|_\mathfrak{p})$ consists of a pair $(\mathfrak{p}_t, \mathfrak{q}_t)$ of smooth curves
\begin{align*}
\mathfrak{p}_t \in C^\infty ([0,1], \mathrm{Gr}_d (\mathfrak{g})) ~~~ \text{ and } ~~~
\mathfrak{q}_t \in C^\infty ([0,1], \mathrm{Gr}_l (\mathfrak{h}))
\end{align*}
such that for all $t$,
\begin{itemize}
\item[(i)] $\mathfrak{p}_t$ is a Lie subalgebra of $\mathfrak{g}$ with $\mathfrak{p}_0 = \mathfrak{p}$,
\item[(ii)] $\mathfrak{q}_t$ is a Lie subalgebra of $\mathfrak{h}$ with $\mathfrak{q}_0 = \mathfrak{q}$,
\item[(iii)] $\phi ( \mathfrak{p}_t )\subset \mathfrak{q}_t$ and the map $\phi|_{\mathfrak{p}_t} : \mathfrak{p}_t \rightarrow \mathfrak{q}_t$ is a Lie algebra homomorphism.
\end{itemize}
\end{defn}

Let $G$ and $H$ be the unique simply connected Lie groups integrating the Lie algebras $\mathfrak{g}$ and $\mathfrak{h}$, respectively. The Lie algebra homomorphism $\phi$ then integrates to a Lie group homomorphism $\Phi : G \rightarrow H$.

\begin{defn}\label{defn-eq-sub}
Two deformations $(\mathfrak{p}_t, \mathfrak{q}_t)$ and $(\mathfrak{p}'_t, \mathfrak{q}'_t)$ are said to be equivalent if there exists a smooth curve $g_t$ in $G$ starting at the identity $1_G \in G$ such that $\mathfrak{p}_t' = \mathrm{Ad}^G_{g_t} \mathfrak{p}_t$ and $\mathfrak{q}_t' = \mathrm{Ad}^G_{\Phi(g_t)} \mathfrak{q}_t$.
\end{defn}

\begin{thm}
Let  $(\mathfrak{p}_t, \mathfrak{q}_t)$ be a deformation of a morphism Lie subalgebra  $(\mathfrak{p}, \mathfrak{q}, \phi|_\mathfrak{p})$. Then 
\begin{align*}
(\dot {\mathfrak{p}}_0 ,  \dot {\mathfrak{q}}_0 , 0) \in \mathrm{Hom} (\mathfrak{p}, \mathfrak{g}/ \mathfrak{p}) \oplus \mathrm{Hom} (\mathfrak{q}, \mathfrak{h}/ \mathfrak{q})  \oplus  \mathfrak{h}/ \mathfrak{q}  = C^1_\mathsf{mLA} (\varphi, \phi/ \varphi)
\end{align*}
is a $1$-cocycle, i.e. lies in $Z^1_\mathsf{mLA} (\varphi, \phi/ \varphi)$. Moreover, the corresponding cohomology class in $H^1_\mathsf{mLA} (\varphi, \phi/ \varphi)$ depends only on the equivalence class of the deformation.
\end{thm}

\begin{proof}
Let $a_t$ be a smooth curve in $\mathrm{GL} (\mathfrak{g})$ starting at the identity such that $a_t (\mathfrak{p}) = \mathfrak{p}_t$, and $b_t$ be a smooth curve in $\mathrm{GL} (\mathfrak{h})$ such that $b_t (\mathfrak{q}) = \mathfrak{q}_t$. We may choose $a_t$ and $b_t$ in such a way that $\phi (a_t (p)) = b_t (\phi(p))$, for $p \in \mathfrak{p}$ and for all $t$. Let $\overline{a_t} : \mathfrak{g} / \mathfrak{p} \rightarrow \mathfrak{g} / \mathfrak{p}_t$ and $\overline{b_t} : \mathfrak{h} / \mathfrak{q} \rightarrow \mathfrak{h} / \mathfrak{q}_t$ be the induced isomorphisms. For each $t$, we define a map $\sigma_t \in \mathrm{Hom} (\wedge^2 \mathfrak{p}, \mathfrak{g}/ \mathfrak{p})$ by
\begin{align*}
\sigma_t (p,r) := \overline{a_t}^{-1} (  [a_t (p), a_t (r)]_\mathfrak{g} ~\text{mod } \mathfrak{p}_t) = (a_t ^{-1} \circ [a_t (p), a_t (r)]_\mathfrak{g} ) \text{ mod }\mathfrak{p},~ \text{ for } p, r \in \mathfrak{p}.
\end{align*}
Since $\mathfrak{p}_t$ is a Lie subalgebra of $\mathfrak{g}$, we have $\sigma_t = 0$, for all $t$. By differentiating $\sigma_t = 0$ at $t=0$, we obtain
\begin{align*}
(- \dot a_0 [p,r]_\mathfrak{g} + [\dot a_0 (p), r]_\mathfrak{g} + [p, \dot a_0 (r)]_\mathfrak{g} ) \text{ mod } \mathfrak{p} = 0.
\end{align*}
In view of (\ref{p-der}),  the above identity is same as $\delta'_\mathsf{CE} ( \dot {\mathfrak{p}}_0) = 0$. Similarly, one can show that $\delta''_\mathsf{CE} ( \dot {\mathfrak{q}}_0) = 0$. Moreover, for any $p \in \mathfrak{p}$, we have
\begin{align*}
((\phi / \varphi )\circ \dot {\mathfrak{p}}_0 - \dot {\mathfrak{q}}_0 \circ \varphi)(p) =~& ( \phi / \varphi) (\frac{d}{dt}|_{t=0} a_t (p) ~ \text{ mod } \mathfrak{p}) - \frac{d}{dt}|_{t=0} b_t (\phi (p)) ~\text{ mod } \mathfrak{q} \\
=~& \frac{d}{dt}|_{t=0} \big( \phi (a_t (p)) - b_t (\phi(p)) \big)~ \text{ mod } \mathfrak{q} ~= 0.
\end{align*}
Hence $\delta_\mathsf{mLA} ( \dot {\mathfrak{p}}_0, \dot {\mathfrak{q}}_0 , 0 ) = (\delta'_\mathsf{CE} (\dot {\mathfrak{p}}_0),~ \delta''_\mathsf{CE} (\dot {\mathfrak{q}}_0), ~ (\phi / \varphi ) \circ \dot {\mathfrak{p}}_0 - \dot {\mathfrak{q}}_0 \circ \varphi) = 0$.

Next, let $(\mathfrak{p}_t, \mathfrak{q}_t)$ and $(\mathfrak{p}'_t, \mathfrak{q}'_t)$ be two equivalent deformations as of Definition \ref{defn-eq-sub}. Then $a_t' = \mathrm{Ad}^G_{g_t} \circ a_t$ is a curve in $\mathrm{GL}(\mathfrak{g})$ and $b_t' = \mathrm{Ad}^H_{\Phi(g_t)} \circ b_t$ is a curve in $\mathrm{GL}(\mathfrak{h})$ satisfying $a_t' (\mathfrak{p}) = \mathfrak{p}'_t$ and $b_t' (\mathfrak{q}) = \mathfrak{q}'_t$. By differentiating 
$a_t' = \mathrm{Ad}^G_{g_t} \circ a_t$ and  $b_t' = \mathrm{Ad}^H_{\Phi(g_t)} \circ b_t$ at $t=0$ and taking quotient by $\mathfrak{p}$ and $\mathfrak{q}$, respectively, we get
\begin{align*}
\dot {\mathfrak{p}}_0' (p) = \dot {\mathfrak{p}}_0 (p) + [\dot g_0, p]_\mathfrak{g} ~ \text{ mod } \mathfrak{p}  ~~~ \text{ and } ~~~
\dot {\mathfrak{q}}_0' (q) = \dot {\mathfrak{q}}_0 (q) + [\phi (\dot g_0), q]_\mathfrak{h} ~ \text{ mod } \mathfrak{q}, ~\text{ for } p \in \mathfrak{p}, q \in \mathfrak{q}.
\end{align*}
This shows that $  ( \dot {\mathfrak{p}}_0, \dot {\mathfrak{q}}_0 , 0 ) - ( \dot {\mathfrak{p}}'_0, \dot {\mathfrak{q}}'_0 , 0 ) = \delta_\mathsf{mLA} ( \dot g_0)$ is a coboundary. Hence the proof.
\end{proof}

\section{Abelian extensions of morphism Lie algebras}\label{sec-5}
In this section, we study abelian extensions of morphism Lie algebras. Our main result says that isomorphism classes of abelian extensions are classified by the `simple' second cohomology group of morphism Lie algebras.

Let $(\mathfrak{g}, \mathfrak{h}, \phi)$ be a morphism Lie algebra and $(V, W, \psi)$ be a representation of it. A $n$-coboundary in $B^n_\mathsf{mLA}(\phi, \psi)$ is said to be `simple' if it is of the form $\delta_\mathsf{mLA}(d_0, \delta_0, 0)$, for some $(d_0, \delta_0, 0)\in C^{n-1}_\mathsf{mLA}(\phi, \psi)$. The space of all simple $n$-coboundaries are denoted by $B^n_{\mathsf{mLA},s}(\phi, \psi)$. The quotients $ H^n_{\mathsf{mLA},s} := \frac{ Z^n_\mathsf{mLA}(\phi, \psi)  }{  B^n_{\mathsf{mLA},s} (\phi, \psi) },$ for $n \geq 0$, are called the simple cohomology groups of the morphism Lie algebra $(\mathfrak{g}, \mathfrak{h}, \phi)$ with coefficients in $(V, W, \psi).$

\medskip

Let $(\mathfrak{g}, \mathfrak{h}, \phi)$ be a morphism Lie algebra. Let $(V, W , \psi)$ be a triple consisting of two vector spaces $V,W$ and a linear map $\psi : V \rightarrow W$. Note that the triple $(V, W, \psi)$ can be regarded as a morphism Lie algebra with abelian Lie brackets on $V$ and $W$.

\begin{defn}\label{abelian-defn}
(i) An abelian extension of a morphism Lie algebra $(\mathfrak{g}, \mathfrak{h}, \phi)$ by a triple $(V,W, \psi)$ is a short exact sequence of morphism Lie algebras
\begin{align}\label{abelian-diagram}
\xymatrix{
0 \ar[r] & V \ar[r]^i \ar[d]_\psi & \widehat{\mathfrak{g}} \ar[d]^{\widehat{\phi}} \ar[r]^p & \mathfrak{g} \ar[d]^\phi \ar[r]  & 0 \\
0 \ar[r] & W \ar[r]_{\overline{i}} & \widehat{\mathfrak{h}} \ar[r]_{\overline{p}} & \mathfrak{h} \ar[r]  & 0.
}
\end{align}

(ii) Two abelian extensions $(\widehat{\mathfrak{g}}, \widehat{\mathfrak{h}}, \widehat{\phi})$ and $(\widehat{\mathfrak{g}}', \widehat{\mathfrak{h}}', \widehat{\phi}' )$ are said to be isomorphic if there is an isomorphism $(\alpha, \beta)$ of morphism Lie algebras from $(\widehat{\mathfrak{g}}, \widehat{\mathfrak{h}}, \widehat{\phi})$ to $(\widehat{\mathfrak{g}}', \widehat{\mathfrak{h}}', \widehat{\phi}' )$ making the following diagram commutative
\begin{align}
\xymatrixrowsep{0.36cm}
\xymatrixcolsep{0.36cm}
\xymatrix{
0  \ar[rr] & & V \ar[rr]  \ar@{=}[rd] \ar[dd] & & \widehat{\mathfrak{g}} \ar[rd]^{\alpha} \ar[dd] \ar[rr] & & \mathfrak{g} \ar[dd] \ar@{=}[rd] \ar[rr]  & & 0 & \\
 & 0 \ar[rr] & & V \ar[rr] \ar[dd] & & \widehat{\mathfrak{g}'} \ar[dd] \ar[rr] & & \mathfrak{g} \ar[dd] \ar[rr]  & & 0 \\
0 \ar[rr] & & W \ar[rr] \ar@{=}[rd] & & \widehat{\mathfrak{h}} \ar[rd]_\beta \ar[rr] & & \mathfrak{h} \ar@{=}[rd] \ar[rr]  & & 0 & \\
 & 0 \ar[rr] & & W \ar[rr] & & \widehat{\mathfrak{h}'} \ar[rr] & & \mathfrak{h} \ar[rr]  & & 0 .
}
\end{align}
\end{defn}

Let (\ref{abelian-diagram}) be an abelian extension of the morphism Lie algebra $(\mathfrak{g}, \mathfrak{h}, \phi)$ by the triple $(V,W, \psi)$. A section of this abelian extension is a pair $(s, \overline{s})$ of linear maps $s : \mathfrak{g} \rightarrow \widehat{\mathfrak{g}}$ and $\overline{s} : \mathfrak{h} \rightarrow \widehat{\mathfrak{h}}$ such that $p \circ s = \mathrm{id}_\mathfrak{g}$ and $\overline{p} \circ \overline{s} = \mathrm{id}_\mathfrak{h}$. Note that, a section always exists.

Let $(s, \overline{s})$ be a section. Then there is a representation of $\mathfrak{g}$ on the vector space $V$ given by $\rho_V : \mathfrak{g} \rightarrow \mathrm{End}(V)$, $\rho_V(x) v := [s(x), i(v)]_{\widehat{\mathfrak{g}}}$, for $x \in \mathfrak{g}$ and $v \in V$. Note that the right hand side of the above equality lies in $\mathrm{ker }~p = \mathrm{im }~i$. Hence it can be considered as an element in $V$ via the embedding $i: V \rightarrow \widehat{\mathfrak{g}}$. Similarly, there is a representation of $\mathfrak{h}$ on the vector space $W$ given by $\rho_W : \mathfrak{h} \rightarrow \mathrm{End}(W)$, $\rho_W (h) w := [\overline{s}(h), \overline{i}(w)]_{\widehat{\mathfrak{h}}}$, for $h \in \mathfrak{g}$ and $w \in W$. Finally, we observe that
\begin{align*}
&\psi (\rho_V (x) v) - \rho_W (\phi (x)) \psi (v) \\
&= \widehat{ \phi} [s(x), i (v)]_{\widehat{ \mathfrak{g}}} - [\overline{s} \circ \phi (x) , \overline{i} \circ \psi (v)]_{\widehat{ \mathfrak{h}}} \\
&= [\widehat{ \phi} \circ s(x) - \overline{s} \circ \phi (x),~ \overline{i} \circ \psi (v)]_{\widehat{\mathfrak{h}}} = 0 ~~~ (\text{as } \widehat{ \phi} \circ s(x) - \overline{s} \circ \phi (x) \in \mathrm{ker }~ \overline{p} = \mathrm{im }~ \overline{i}).
\end{align*}
This shows that $(V,W,\psi)$ is a representation of the morphism Lie algebra $(\mathfrak{g}, \mathfrak{h}, \phi)$. This representation is infact independent of the choice of the section. To prove this, take another section $(s', \overline{s}')$ of the abelian extension (\ref{abelian-diagram}). We observe that
\begin{align*}
s(x) - s'(x) \in \mathrm{ker }~ p = \mathrm{im }~i ~~~ \text{ and } ~~~  \overline{s}(h) - \overline{s}'(h) \in \mathrm{ker }~ \overline{p} = \mathrm{im }~\overline{i}.
\end{align*}
Hence ~$\rho_V(x) v - \rho'_V (x) v = [s(x) - s'(x), i(v)]_{\widehat{\mathfrak{g}}} = 0$~ and ~$\rho_W (h) w - \rho'_W (h) w := [\overline{s}(h) - \overline{s}'(h), \overline{i}(w)]_{\widehat{\mathfrak{h}}} = 0$. Here $\rho'_V$ (resp. $\rho'_W$) is the representation of the Lie algebra $\mathfrak{g}$ (resp. $\mathfrak{h}$) on the vector space $V$ (resp. $W$) induced by the section $(s', \overline{s}').$

\medskip

Let $(\mathfrak{g}, \mathfrak{h}, \phi)$ be a morphism Lie algebra and $(V,W,\psi)$ be a fixed representation of it. Let $\mathrm{Ext}(\phi, \psi)$ denote the isomorphism classes of abelian extensions of $(\mathfrak{g}, \mathfrak{h}, \phi)$ by the triple $(V,W,\psi)$ so that the induced representation is the prescribed one.

\begin{thm}
Let $(\mathfrak{g}, \mathfrak{h}, \phi)$ be a morphism Lie algebra and $(V,W,\psi)$ be a representation of it. Then there is a one-to-one correspondence between $\mathrm{Ext}(\phi, \psi)$ and the simple second cohomology group $H^2_{\mathsf{mLA},s} (\phi, \psi).$
\end{thm}

\begin{proof}
Let (\ref{abelian-diagram}) be an abelian extension of the morphism Lie algebra $(\mathfrak{g}, \mathfrak{h}, \phi)$ by the triple $(V,W,\psi)$. Let $(s, \overline{s})$ be a section of it. We define maps
\begin{align*}
&\theta \in \mathrm{Hom} (\wedge^2 \mathfrak{g}, V), ~\theta (x, y) := [s(x), s(y)]_{\widehat{\mathfrak{g}}} - s [x, y]_\mathfrak{g}, \\
&\gamma \in \mathrm{Hom} (\wedge^2 \mathfrak{h}, W), ~\theta (h,k) := [\overline{s}(h), \overline{s}(k)]_{\widehat{\mathfrak{h}}} - \overline{s} [h,k]_\mathfrak{h},\\
&\eta \in \mathrm{Hom} (\mathfrak{g},W),~ \eta (x) := \widehat{\phi} (s(x)) - \overline{s}(\phi (x)).
\end{align*}
Thus, we have $(\theta, \gamma, \eta) \in C^2_\mathsf{mLA}(\phi, \psi).$ It is known from the classical theory of abelian extensions of Lie algebras \cite{weibel} that $\delta'_\mathsf{CE} (\theta) = 0$ and $\delta''_\mathsf{CE} (\gamma) = 0$. Moreover, by a straightforward computation we have $(\psi \circ \theta - \gamma \circ \wedge^2 \phi - \delta'''_\mathsf{CE} (\eta)) = 0$. Hence we have $(\theta, \gamma, \eta) \in Z^2_\mathsf{mLA}(\phi, \psi)$. Therefore $(\theta, \gamma, \eta)$ corresponds to a cohomology class in $H^2_{\mathsf{mLA},s} (\phi, \psi)$. It is easy to verify that this cohomology class does not depend on the choice of section. Next, we let $(\widehat{\mathfrak{g}}, \widehat{\mathfrak{h}}, \widehat{\phi})$ and $(\widehat{\mathfrak{g}}', \widehat{\mathfrak{h}}', \widehat{\phi}')$ be two isomorphic abelian extensions and the isomorphism is given by $(\alpha, \beta)$ (see Definition \ref{abelian-defn} (ii)). Let $(s, \overline{s})$ be a section of the first abelian extension. We have $p' \circ (\alpha \circ s) = p \circ s = \mathrm{id}_\mathfrak{g}$ and $\overline{p}' \circ (\beta \circ \overline{s}) = \overline{p} \circ \overline{s} = \mathrm{id}_\mathfrak{h}$, which shows that $(\alpha \circ s, \beta \circ \overline{s})$ is a section of the second abelian extension. If $(\theta', \gamma', \eta') \in Z^2_\mathsf{mLA} (\phi, \psi)$ is the $2$-cocycle corresponding to the second abelian extension and its section $(\alpha \circ s, \beta \circ \overline{s})$ then
\begin{align*}
\theta' (x,y) =~& [(\alpha \circ s)(x), (\alpha \circ s)(y)]_{\widehat{\mathfrak{g}}'} - (\alpha \circ s) ([x, y]_\mathfrak{g}) \\
=~& \alpha ([s(x), s(y)]_{\widehat{\mathfrak{g}}} - s[x, y]_\mathfrak{g}) = \theta(x, y) \quad (\text{as } \alpha|_V = \mathrm{id}_V).
\end{align*} 
Similarly, one can show that $\gamma' = \gamma$ and $\eta' = \eta$. Hence $(\theta, \gamma, \eta) = (\theta', \gamma', \eta')$. This shows that there is a well-defined map $\Theta_1 : \mathrm{Ext} (\phi, \psi) \rightarrow H^2_{\mathsf{mLA},s} (\phi, \psi)$.

\medskip

Conversely, let $(\theta, \gamma, \eta) \in Z^2_\mathsf{mLA}(\phi, \psi)$ be a $2$-cocycle. Take $\widehat{\mathfrak{g}} = \mathfrak{g} \oplus V$ and $\widehat{\mathfrak{h}} = \mathfrak{h} \oplus W$, and define a map $\widehat{\phi} : \widehat{\mathfrak{g}} \rightarrow \widehat{\mathfrak{h}}$ by $\widehat{ \phi } (x, v) = (\phi (x),~\psi (v) + \eta (x))$, for $(x, v) \in \widehat{\mathfrak{g}}$. We also define bilinear skew-symmetric brackets on $\widehat{\mathfrak{g}}$ and $\widehat{\mathfrak{h}}$ by
\begin{align*}
[(x, v), (x', v')]_{\widehat{\mathfrak{g}}} :=~& ([x, x']_\mathfrak{g},~ \rho_V (x) v' - \rho_V (x') v + \theta (x, x')),\\
[(h, w), (h', w')]_{\widehat{\mathfrak{h}}} :=~& ([h, h']_\mathfrak{h}, ~ \rho_W (h) w' - \rho_W (h') w + \gamma (h, h')),
\end{align*}
for $(x, v), (x', v') \in \widehat{\mathfrak{g}}$ and $(h, w), (h', w') \in \widehat{\mathfrak{h}}$. Since $(\theta, \gamma, \eta)$ is a $2$-cocycle, it follows that $(\widehat{\mathfrak{g}}, \widehat{\mathfrak{h}}, \widehat{\phi})$ is a morphism Lie algebra with the above (Lie) brackets on $\widehat{\mathfrak{g}}$ and $\widehat{\mathfrak{h}}$, respectively. Moreover, this is an abelian extension of the morphism Lie algebra $(\mathfrak{g}, \mathfrak{h}, \phi)$ by the triple $(V, W, \psi)$. Finally, let $(\theta, \gamma, \eta)$ and $(\theta', \gamma', \eta')$ be two $2$-cocycles such that $(\theta, \gamma, \eta) - (\theta', \gamma', \eta') = \delta_\mathsf{mLA}(d_0, \delta_0, 0)$, for some  $(d_0, \delta_0, 0) \in C^1_\mathsf{mLA}(\phi, \psi)$. Then the abelian extensions  $(\widehat{\mathfrak{g}}, \widehat{\mathfrak{h}}, \widehat{\phi})$ and  $(\widehat{\mathfrak{g}}', \widehat{\mathfrak{h}}', \widehat{\phi}')$ are isomorphic via the map $(\alpha, \beta)$, where $\alpha : \widehat{\mathfrak{g}} \rightarrow \widehat{\mathfrak{g}}'$ and $\beta : \widehat{\mathfrak{h}} \rightarrow \widehat{\mathfrak{h}}'$ are the maps
\begin{align*}
\alpha (x, v) = (x, v + d_0(x)) \quad \text{ and } \quad \beta (h, w) = (h, w+ \delta_0 (h)).
\end{align*}
In other words, there is a well-defined map $\Theta_2 : H^2_{\mathsf{mLA},s} (\phi, \psi) \rightarrow \mathrm{Ext}(\phi, \psi)$. Finally, the maps $\Theta_1$ and $\Theta_2$ are inverses to each other. This completes the proof.
\end{proof}

\section{Classification of skeletal morphism sh Lie algebras}\label{sec-6}

A morphism sh Lie algebra is a triple $(\mathcal{G}, \mathcal{H}, \Phi)$ consisting of two sh Lie algebras $\mathcal{G}, \mathcal{H}$ and a homomorphism $\Phi : \mathcal{G} \rightarrow \mathcal{H}$ of sh Lie algebras. In this section, we focus on morphism sh Lie algebras whose underlying sh Lie algebras are concentrated in degrees $0$ and $1$. We first recall some definitions from \cite{baez-crans}.

\begin{defn}
A $2$-term sh Lie algebra $\mathcal{G}$ consists of a chain complex $\mathfrak{g}_1 \xrightarrow{d} \mathfrak{g}_0$ together with a skew-symmetric bilinear map $l_2 : \mathfrak{g}_i \otimes \mathfrak{g}_j \rightarrow \mathfrak{g}_{i+j}$ $(0 \leq i+j \leq 1)$ and a skew-symmetric trilinear map $l_3 : \mathfrak{g}_0 \otimes \mathfrak{g}_0 \otimes \mathfrak{g}_0 \rightarrow \mathfrak{g}_1$ satisfying the following identities
\begin{itemize}
\item[(i)] $dl_2 (x, p) = l_2 (x, dp),$
\item[(ii)] $l_2 (dp, q) = l_2 (p, dq),$
\item[(iii)] $dl_3 (x, y, z) = l_2 (x, l_2 (y, z)) + l_2 (y, l_2 (z,x)) + l_2 (z, l_2 (x,y)),$
\item[(iv)] $l_3 (x, y, dp) = l_2 (x, l_2 (y, p)) + l_2 (y, l_2 (p, x)) + l_2 (p, l_2(x,y)),$
\item[(v)] $l_2 (x, l_3 (y, z, t)) - l_2 (y, l_3 (x,z,t)) + l_2 (z, l_3 (x,y,t)) - l_2 (t, l_3(x,y,z)) 
- l_3 (l_2 (x,y), z, t) \\+ l_3 (l_2 (x, z), y, t) - l_3 (l_2 (x, t), y, z)  - l_3 (l_2 (y,z), x, t) + l_3 (l_2 (y, t), x, z) - l_3 (l_2 (z, t), x, y) = 0,$
\end{itemize}
for $x, y, z, t \in \mathfrak{g}_0$ and $p, q \in \mathfrak{g}_1$. We denote a $2$-term sh Lie algebra by the tuple $\mathcal{G} = (  \mathfrak{g}_1 \xrightarrow{d} \mathfrak{g}_0, l_2, l_3 ).$
\end{defn}

\begin{defn}\label{sh-homo}
Let $\mathcal{G} = (  \mathfrak{g}_1 \xrightarrow{d} \mathfrak{g}_0, l_2, l_3 )$ and $\mathcal{H} = (  \mathfrak{h}_1 \xrightarrow{d'} \mathfrak{h}_0, l'_2, l'_3 )$ be $2$-term sh Lie algebras. A homomorphism $\Phi : \mathcal{G} \rightarrow \mathcal{H}$ is a tuple $\Phi = (\phi_0, \phi_1, \phi_2)$ of linear maps $\phi_0 : \mathfrak{g}_0 \rightarrow \mathfrak{h}_0$, $\phi_1 : \mathfrak{g}_1 \rightarrow \mathfrak{h}_1$ and a skew-symmetric bilinear map $\phi_2 : \mathfrak{g}_0 \otimes \mathfrak{g}_0 \rightarrow \mathfrak{h}_1$ satisfying for $x, y, z \in \mathfrak{g}_0$ and $p \in \mathfrak{g}_1$,
\begin{itemize}
\item[(i)] $\phi_0 \circ d = d' \circ \phi_1$,
\item[(ii)] $d \phi_2 (x,y) = \phi_0 (l_2 (x, y)) - l_2' (\phi_0 (x), \phi_0 (y)),$
\item[(iii)] $\phi_2 (x, d p) = \phi_1 (l_2 (x, p)) - l_2' (\phi_0 (x), \phi_1 (p)),$
\item[(iv)] $l_2' (\phi_0 (x), \phi_2 (y, z)) + c.p. + \phi_2 (x, l_2 (y, z)) + c. p. = \phi_1 (l_3 (x, y, z)) - l_3' (\phi_0 (x), \phi_0 (y), \phi_0 (z)).$
\end{itemize} 
\end{defn}

\begin{defn}
A $2$-term morphism sh Lie algebra is a triple $(\mathcal{G}, \mathcal{H}, \Phi)$ in which $\mathcal{G} =  (  \mathfrak{g}_1 \xrightarrow{d} \mathfrak{g}_0, l_2, l_3 )$ and  $\mathcal{H} = (  \mathfrak{h}_1 \xrightarrow{d'} \mathfrak{h}_0, l'_2, l'_3 )$ are $2$-term sh Lie algebras and $\Phi = (\phi_0, \phi_1, \phi_2) : \mathcal{G} \rightarrow \mathcal{H}$ is a homomorphism between them.
\end{defn}

A skeletal morphism sh Lie algebra is a $2$-term morphism sh Lie algebra $(\mathcal{G}, \mathcal{H}, \Phi)$ in which $d= 0$ and $d' = 0$. In other words, the $2$-term sh Lie algebras $\mathcal{G}$ and $\mathcal{H}$ are skeletal in the sense of \cite{baez-crans}. 

\medskip

Let $(\mathcal{G} = (\mathfrak{g}_1 \xrightarrow{0} \mathfrak{g}_0, l_2, l_3) ,~ \mathcal{H} = (\mathfrak{h}_1 \xrightarrow{0} \mathfrak{h}_0, l'_2, l'_3),~ \Phi = (\phi_0, \phi_1, \phi_2))$ be a skeletal morphism sh Lie algebra. Suppose there are skew-symmetric bilinear maps $\sigma : \mathfrak{g}_0 \otimes \mathfrak{g}_0 \rightarrow \mathfrak{g}_1$ and $\sigma' : \mathfrak{h}_0 \otimes \mathfrak{h}_0 \rightarrow \mathfrak{h}_1$, and a linear map $\phi : \mathfrak{g}_0 \rightarrow \mathfrak{h}_1$. We define two skew-symmetric trilinear  maps and a skew-symmetric bilinear map
$$\underline{l_3} : \mathfrak{g}_0 \otimes \mathfrak{g}_0 \otimes \mathfrak{g}_0 \rightarrow \mathfrak{g}_1, \qquad
\underline{l'_3} : \mathfrak{h}_0 \otimes \mathfrak{h}_0 \otimes \mathfrak{h}_0 \rightarrow \mathfrak{h}_1 \quad ~~~ \text{ and } ~~~
\underline{\phi_2} : \mathfrak{g}_0 \otimes \mathfrak{g}_0 \rightarrow \mathfrak{h}_1$$
by
\begin{align}
&\underline{l_3} (x, y, z) := l_3 (x, y, z) + \{ l_2 (x, \sigma (y, z)) + c.p. \} + \{ \sigma (x, l_2 (y, z)) + c. p. \}, \label{sk-eq1}\\
&\underline{l'_3} (h, j, k) := l'_3 (h, j, k) + \{ l'_2 (h, \sigma' (j, k)) + c. p. \} + \{   \sigma' (h, l'_2 (j, k)) + c. p. \}, \label{sk-eq2}\\
&\underline{\phi_2} (x, y) := \phi_2 (x, y) + \phi_1 (\sigma (x,y)) - \sigma' (\phi_0 (x), \phi_0 (y)) - l_2' (\phi_0 (x), \phi (y)) - l_2' (\phi(x), \phi_0 (y)) + \phi (l_2 (x, y)), \label{sk-eq3}
\end{align}
for $x, y, z \in \mathfrak{g}_0$ and $h, j, k \in \mathfrak{h}_0$. Then it is easy to see that $$(\underline{\mathcal{G}} = (\mathfrak{g}_1 \xrightarrow{0} \mathfrak{g}_0, l_2, \underline{l_3}) ,~ \underline{\mathcal{H}} = (\mathfrak{h}_1 \xrightarrow{0} \mathfrak{h}_0, l'_2, \underline{l'_3}),~ \underline{\Phi} = (\phi_0, \phi_1, \underline{\phi_2}))$$
is a skeletal morphism sh Lie algebra. In this case, we say that the skeletal morphism sh Lie algebras $(\mathcal{G}, \mathcal{H}, \Phi)$ and $(\underline{\mathcal{G}}, \underline{\mathcal{H}}, \underline{\Phi})$ are equivalent.
In the following result, we show that (equivalence classes of) skeletal morphism sh Lie algebras are closely related to the cohomology of morphism Lie algebras introduced in Section \ref{section-cohomology}.

\begin{thm}
There is a one-to-one correspondence between skeletal morphism sh Lie algebras and triples of the form $( (\mathfrak{g}, \mathfrak{h}, \phi), (V,W,\psi), (\theta, \gamma, \eta))$ in which $(\mathfrak{g}, \mathfrak{h}, \phi)$ is a morphism Lie algebra, $(V,W, \psi)$ is a representation and $(\theta, \gamma, \eta) \in Z^3_\mathsf{mLA} (\phi, \psi)$ is a $3$-cocycle.

Moreover, it extends to a one-to-one correspondence between  equivalence classes of skeletal morphism sh Lie algebras and triples $( (\mathfrak{g}, \mathfrak{h}, \phi), (V,W,\psi), [(\theta, \gamma, \eta)])$, where $[(\theta, \gamma, \eta)] \in H^3_\mathsf{mLA}(\phi, \psi)$.
\end{thm}

\begin{proof}
Let $(\mathcal{G}, \mathcal{H}, \Phi)$ be a skeletal morphism sh Lie algebra, where $\mathcal{G} = (\mathfrak{g}_1 \xrightarrow{0} \mathfrak{g}_0, l_2, l_3)$, $\mathcal{H} = (\mathfrak{h}_1 \xrightarrow{0} \mathfrak{h}_0, l'_2, l'_3)$ and $\Phi = (\phi_0, \phi_1, \phi_2)$. Since $\mathcal{G} = (\mathfrak{g}_1 \xrightarrow{0} \mathfrak{g}_0, l_2, l_3)$ is a skeletal $2$-term sh Lie algebra, it follows from \cite[Theorem 6.7]{baez-crans} that $\mathfrak{g}_0$ is a Lie algebra with the Lie bracket $[x,y]_{\mathfrak{g}_0} := l_2 (x, y)$, for $x, y \in \mathfrak{g}_0$, and $\mathfrak{g}_1$ is a representation of the Lie algebra $\mathfrak{g}_0$ with the action map $\rho_{\mathfrak{g}_1} : \mathfrak{g}_0 \rightarrow \mathrm{End}(\mathfrak{g}_1)$ given by $\rho_{\mathfrak{g}_1}(x)(p) := l_2 (x, p)$, for $x \in \mathfrak{g}_0$, $p \in \mathfrak{g}_1$. Moreover, the skew-symmetric trilinear map $l_3 \in \mathrm{Hom}(\wedge^3 \mathfrak{g}_0, \mathfrak{g}_1)$ is a $3$-cocycle in the Chevalley-Eilenberg cohomology complex of $\mathfrak{g}_0$ with coefficients in the representation $\mathfrak{g}_1$. Similarly, since $\mathcal{H} = (\mathfrak{h}_1 \xrightarrow{0} \mathfrak{h}_0, l'_2, l'_3)$ is a skeletal $2$-term sh Lie algebra, we have that $\mathfrak{h}_0$ is a Lie algebra with the bracket $[h,k]_{\mathfrak{h}_0} := l'_2 (h,k)$, for $h, k \in \mathfrak{h}_0$; the space $\mathfrak{h}_1$ is a representation of the Lie algebra $\mathfrak{h}_0$ with the action map $\rho_{\mathfrak{h}_1} : \mathfrak{h}_0 \rightarrow \mathrm{End}(\mathfrak{h}_1)$, $\rho_{\mathfrak{h}_1} (h) (r) := l'_2 (h, r)$, for $h \in \mathfrak{h}_0, r \in \mathfrak{h}_1$; and $l'_3 \in \mathrm{Hom}(\wedge^3 \mathfrak{h}_0, \mathfrak{h}_1)$ is a $3$-cocycle in the Chevalley-Eilenberg complex of $\mathfrak{h}_0$ with coefficients in $\mathfrak{h}_1$. Finally, since $\Phi = (\phi_0, \phi_1, \phi_2) : \mathcal{G} \rightarrow \mathcal{H}$ is a homomorphism of $2$-term sh Lie algebras, we have from Definition \ref{sh-homo} that
\begin{align}
&\phi_0 ([x, y]_{\mathfrak{g}_0}) = [\phi_0 (x), \phi_0 (y)]_{\mathfrak{h}_0}, \label{sk1}\\
&\phi_1 (\rho_{\mathfrak{g}_1} (x)(p) ) = \rho_{\mathfrak{h}_1} (\phi_0 (x)) \phi_1(p), \label{sk2}\\
&\rho_{\mathfrak{h}_1} (\phi_0 (x)) \phi_2 (y, z) + c. p. + \phi_2 (x, [y, z]_{\mathfrak{g}_0}) + c. p. = (\phi_1 \circ l_3 - l'_3 \circ \wedge^3 \phi_0) (x, y, z), \label{sk3}
\end{align}
for $x, y, z \in \mathfrak{g}_0$ and $p \in \mathfrak{g}_1$. The condition (\ref{sk1}) says that $\phi_0 : \mathfrak{g}_0 \rightarrow \mathfrak{h}_0$ is a Lie algebra homomorphism. In other words, $(\mathfrak{g}_0, \mathfrak{h}_0, \phi_0)$ is a morphism Lie algebra. The condition (\ref{sk2}) says that the triple $(\mathfrak{g}_1, \mathfrak{h}_1, \phi_1)$ is a representation of the morphism Lie algebra $(\mathfrak{g}_0, \mathfrak{h}_0, \phi_0)$. 

It follows that we may consider the representation of the Lie algebra $\mathfrak{g}_0$ on the vector space $\mathfrak{h}_1$ with the action map $\rho_{\mathfrak{h}_1}^{\mathfrak{g}_0} : \mathfrak{g}_0 \rightarrow \mathrm{End}(\mathfrak{h}_1)$, $\rho_{\mathfrak{h}_1}^{\mathfrak{g}_0} (x)(r) := \rho_{\mathfrak{h}_1} (\phi_0 (x))r = l'_2 (\phi_0 (x), r)$,  for $x \in \mathfrak{g}_0, r \in \mathfrak{h}_1$. Denote this representation by $(\mathfrak{h}_1)_{\phi_0}$. Let $\{ C^\ast_\mathsf{CE} (\mathfrak{g}_0, (\mathfrak{h}_1)_{\phi_0}), \delta'''_\mathsf{CE} \}$ be the corresponding Chevalley-Eilenberg cochain complex. Note that the condition (\ref{sk3}) is equivalent to $(\phi_1 \circ l_3 - l'_3 \circ \wedge^3 \phi_0 - \delta'''_\mathsf{CE} (\phi_2)) = 0$. Thus, combining all these facts, we get $\delta_{\mathsf{mLA}} (l_3, l_3', \phi_2) = 0.$ Hence we get a required triple $((\mathfrak{g}_0, \mathfrak{h}_0, \phi_0), (\mathfrak{g}_1, \mathfrak{h}_1, \phi_1), (l_3, l_3', \phi_2))$.

Finally, if $(\mathcal{G}, \mathcal{H}, \Phi)$ and $(\underline{\mathcal{G}}, \underline{\mathcal{H}}, \underline{\Phi})$  are equivalent skeletal morphism sh Lie algebras, then we have from (\ref{sk-eq1}), (\ref{sk-eq2}) and (\ref{sk-eq3}) that
\begin{align*}
 (\underline{l_3}, \underline{l_3'}, \underline{\phi_2}) - (l_3, l_3', \phi_2) = \delta_\mathsf{mLA} (\sigma, \sigma', \phi).
\end{align*}
Hence $[(l_3, l_3', \phi_2)] = [(\underline{l_3}, \underline{l_3'}, \underline{\phi_2})]$ in $H^3_\mathsf{mLA}(\phi_0, \phi_1)$.

\medskip

Conversely, let  $((\mathfrak{g}, \mathfrak{h}, \phi), (V,W,\psi), (\theta, \gamma, \eta))$ be a triple consisting of a morphism Lie algebra $(\mathfrak{g}, \mathfrak{h}, \phi)$, a representation $(V,W, \psi)$ and a $3$-cocycle $(\theta, \gamma, \eta) \in Z^3_\mathsf{mLA} (\phi, \psi)$. Then by a straightforward observation one can show that $(\mathcal{G}, \mathcal{H}, \Phi)$ is a skeletal morphism sh Lie algebra, where
\begin{align*}
&\mathcal{G} = ( V \xrightarrow{0} \mathfrak{g}, l_2, \theta),~ \text{ where } l_2 (x, v) = \rho_V (x) v, \text{ for } x \in \mathfrak{g}, v \in V,\\
&\mathcal{H} = ( W \xrightarrow{0} \mathfrak{h}, l_2', \gamma ), ~ \text{ where } l_2' (h, w) = \rho_W (h) w, \text{ for } h \in \mathfrak{h}, w \in W,\\
&\Phi = (\phi, \psi, \eta).
\end{align*}
Moreover, if $(\theta, \gamma, \eta)$ and $(\underline{\theta}, \underline{\gamma}, \underline{\eta})$ are two cohomologous $3$-cocycles, then the corresponding skeletal morphism sh Lie algebras are obviously equivalent. This completes the proof.
\end{proof}

\section{Morphism Lie groups}\label{sec-7}

In this section, we consider morphism Lie groups and define their cohomology with coefficients in a module. We also define differentiable cohomology of morphism Lie groups and construct a map from the differentiable cohomology of a morphism Lie group to the cohomology of the corresponding morphism Lie algebra.

A morphism Lie group is a triple $(G,H, \Phi)$ consisting of two Lie groups $G,H$ and a homomorphism $\Phi : G \rightarrow H$ of Lie groups. Let $(G,H, \Phi)$ be a morphism Lie group. Let $\mathrm{Lie}(G)$ and $\mathrm{Lie}(H)$ be the Lie algebras corresponding to the Lie groups $G$ and $H$, respectively. The Lie group homomorphism $\Phi : G \rightarrow H$ induces a Lie algebra homomorphism $\mathrm{Lie}(\Phi) : \mathrm{Lie}(G) \rightarrow \mathrm{Lie}(H)$ by
\begin{align*}
(\mathrm{Lie}(\Phi))(X) = (T_{1_G} \Phi) (X) , ~\text{ for } X \in \mathrm{Lie}(G) = T_{1_G} (G).
\end{align*}
In other words, the triple $(\mathrm{Lie}(G), \mathrm{Lie}(H), \mathrm{Lie}(\Phi))$ is a morphism Lie algebra, called the `infinitesimal' of the morphism Lie group $(G,H, \Phi)$.

Let $(G, H, \Phi)$ be a morphism Lie group. A module over it consists of a triple $(V,W, \psi)$ in which $V,W$ are vector spaces with Lie group homomorphisms $\varrho_V : G \rightarrow \mathrm{Aut} (V)$ and $\varrho_W : H \rightarrow \mathrm{Aut}(W)$, and $\psi : V \rightarrow W$ is a linear map satisfying $\psi ( \varrho_V (g) v) = \varrho_W (\Phi (g)) \psi (v)$, for $g \in G$, $v \in V$. If $(V,W, \psi)$ is a module over the morphism Lie group $(G,H, \Phi)$, then it can be considered as a representation of the infinitesimal morphism Lie algebra $(\mathrm{Lie}(G), \mathrm{Lie}(H), \mathrm{Lie}(\Phi))$. Here the representation of the Lie algebra $\mathrm{Lie}(G)$ on $V$ is given by
\begin{align*}
\rho_V : \mathrm{Lie} (G) \rightarrow \mathrm{End}(V),~ \rho_V (X) v = \frac{d}{dt}|_{t=0} \varrho_V(\mathrm{exp} (tX))v,  ~\text{ for } X \in \mathrm{Lie}(G), v \in V,
\end{align*}
where $\mathrm{exp} : \mathfrak{g} \rightarrow G$ is the exponential map. The representation of the Lie algebra $\mathrm{Lie}(H)$ on $W$ is defined similarly.

\medskip

In the following, we first recall the group cohomology \cite{weibel} and using it we define cohomology of morphism Lie groups. Let $G$ be a Lie group and $\varrho_V : G \rightarrow \mathrm{Aut}(V)$ be a Lie group homomorphism (in this case we say that $V$ is a $G$-module). Then there is a cochain complex $\{ C^\ast_\mathsf{gp} (G,V), \delta_\mathsf{gp} \}$, where $C^{n \geq 0}_\mathsf{gp} (G,V) = \mathrm{Maps}(G^{\times n}, V)  $ and the coboundary operator $\delta_\mathsf{gp} : C^{n}_\mathsf{gp} (G,V) \rightarrow C^{n +1}_\mathsf{gp} (G,V)$ is given by
\begin{align*}
(\delta_\mathsf{gp} f) (g_1, \ldots, g_{n+1}) =& \varrho_V (g_1)(f (g_2, \ldots, g_{n+1})) + (-1)^{i} f (g_1, \ldots, g_i \cdot g_{i+1}, \ldots, g_{n+1}) + (-1)^{n+1}f (g_1, \ldots, g_n),
\end{align*}
for $f \in C^n_\mathsf{gp}(G,V)$ and $g_1, \ldots, g_{n+1} \in G$.
The corresponding cohomology groups are called the group cohomology of $G$ with coefficients in the $G$-module $V$.
The differential cohomology of $G$ with coefficients in $V$ is the cohomology of the subcomplex $\{ C^\ast_{\mathsf{gp}, d} (G,V), \delta_\mathsf{gp} \}$, where
\begin{align*}
C^n_{\mathsf{gp}, d} (G,V) := \{ f \in C^n_{\mathsf{gp}} (G,V) ~|~ f (g_1, \ldots, g_n) = 0 \text{ if } g_i = 1_G \}.
\end{align*}
The van Est map \cite{van} is a map between cochain complexes
$\mathrm{vE} : C^\ast_{\mathsf{gp}, d} (G,V)  \rightarrow C^\ast_{\mathsf{CE}} (\mathrm{Lie}(G),V)$ 
from $\{ C^\ast_{\mathsf{gp}, d} (G,V), \delta_\mathsf{gp} \}$ to the Chevalley-Eilenberg cochain complex $\{ C^\ast_{\mathsf{CE}} (\mathrm{Lie}(G),V) , \delta_\mathsf{CE} \}.$ It is explicitly given by
\begin{align}\label{ve-map}
(\mathrm{vE})(f) (X_1, \ldots, X_n) = \sum_{\sigma \in S_n} \mathrm{sgn}(\sigma) R_{X_{\sigma (1)}} \cdots R_{X_{\sigma (n)}} (f),
\end{align}
where $R_X : C^n_{\mathsf{gp}, d} (G,V) \rightarrow C^{n-1}_{\mathsf{gp}, d} (G,V)$ is the map which differentiates $f(-, g_2, \ldots, g_n)$ at the point $1_G$ with respect to the right invariant vector field corresponding to $X$.

Let $(G,H, \Phi)$ be a morphism Lie group and $(V,W, \psi)$ be a module over it. Then there is a $G$-module structure on $W$ given by the Lie group homomorphism $\varrho_W^G : G \rightarrow \mathrm{Aut} (W), \varrho_W^G (g) (w) = \varrho_W (\Phi(g)) w,$ for $g \in G, w \in W$. We denote this $G$-module by $W_\Phi$. Let $\delta_\mathsf{gp}' $ (resp.  $\delta_\mathsf{gp}''$) be the coboundary operator for the group cohomology complex of $G$ with coefficients in $V$ (resp. of $H$ with coefficients in $W$), and $\delta'''_\mathsf{gp}$ be the coboundary operator for the group cohomology complex of $G$ with coefficients in $W_\Phi$. For each $n \geq 0$, we define an abelian group $C^n_\mathsf{mLG} (\Phi, \psi)$ by
\begin{align*}
C^0_\mathsf{mLG} (\Phi, \psi) = V  \quad \text{ and } \quad C^{n \geq 1}_\mathsf{mLG} (\Phi, \psi) =  C^n_\mathsf{gp} (G, V) \oplus  C^n_\mathsf{gp} (H, W) \oplus  C^{n-1}_\mathsf{gp} (G, W_\Phi).
\end{align*} 
We define a map $\delta_\mathsf{mLG} : C^n_\mathsf{mLG} (\Phi, \psi) \rightarrow C^{n+1}_\mathsf{mLG} (\Phi, \psi)$, for $n \geq 0$, by
\begin{align*}
\delta_\mathsf{mLG} (v) = (\delta'_\mathsf{gp}(v), ~ \delta''_\mathsf{gp}(\psi(v)) , ~ 0) ~~ \text{ and } ~~ \delta_\mathsf{mLG} (\Theta, \Gamma, \Lambda) = (\delta'_\mathsf{gp} (\Theta),~ \delta''_\mathsf{gp} ( \Gamma), ~ \psi \circ \Theta - \Gamma \circ \Phi^{\times n} - \delta'''_\mathsf{gp} (\Lambda)),
\end{align*}
for $v \in V$ and $(\Theta, \Gamma, \Lambda) \in C^{n \geq 1}_\mathsf{mLG} (\Phi, \psi).$ Then it is easy to see that $(\delta_\mathsf{mLG})^2 = 0$. The cohomology groups of the cochain complex $\{ C^\ast_\mathsf{mLG} (\Phi, \psi), \delta_\mathsf{mLG} \}$ are called the cohomology of the morphism Lie group $(G,H, \Phi)$ with coefficients in the module $(V,W, \psi)$, and denoted by $H^\ast_\mathsf{mLG} (\Phi, \psi).$

\begin{remark}
For any Lie group $G$, the corresponding Lie algebra $\mathrm{Lie}(G)$ can be given a $G$-module structure with the Lie group homomorphism $\rho_{\mathrm{Lie}(G)} : G \rightarrow \mathrm{Aut}(\mathrm{Lie} (G))$, $(\rho_{\mathrm{Lie} (G)} (g)) X  := \mathrm{Ad}^G_g (X)$, for $g \in G$ and $X \in \mathrm{Lie}(G)$. This is called the adjoint $G$-module. It has been shown in \cite{cop} that the group cohomology of $G$ with coefficients in the adjoint $G$-module $\mathrm{Lie}(G)$ controls the deformations of $G$. Similar to the definition of adjoint module of a Lie group, we may define the adjoint module of a morphism Lie group $(G,H, \Phi)$ on the triple $(\mathrm{Lie}(G), \mathrm{Lie}(H), \mathrm{Lie}(\Phi))$. By following \cite{cop} one may show that the cohomology of the morphism Lie group $(G,H,\Phi)$ with coefficients in the adjoint module governs the deformations. We will address deformation related questions for morphism Lie groups in a separate article.
\end{remark}

In the following, we consider the differentiable subcomplex $\{ C^\ast_{\mathsf{mLG},d} (\Phi, \psi), \delta_\mathsf{mLG} \}$ of the cochain complex $\{ C^\ast_\mathsf{mLG} (\Phi, \psi), \delta_\mathsf{mLG} \}$, where
\begin{align*}
C^0_{\mathsf{mLG},d} (\Phi, \psi) = V ~~~ \text{ and } ~~~ C^{n }_{\mathsf{mLG}, d} (\Phi, \psi) := C^n_{\mathsf{gp}, d} (G,V) \oplus C^n_{\mathsf{gp}, d} (H,W) \oplus C^{n-1}_{\mathsf{gp}, d} (G,W_\Phi) \subset C^{n }_\mathsf{mLG} (\Phi, \psi).
\end{align*}
We call the corresponding cohomology groups as the differentiable cohomology of the morphism Lie group $(G,H,\Phi)$ with coefficients in the module $(V,W, \psi)$. We denote cohomology groups by $H^\ast_{\mathsf{mLG},d} (\Phi, \psi)$.

Let $(G,H, \Phi)$ be a morphism Lie group and $(V,W, \psi)$ be a bimodule over it. Therefore, we can consider the differentiable complex $\{ C^\ast_{\mathsf{mLG},d} (\Phi, \psi), \delta_\mathsf{mLG} \}$. On the other hand, the triple $(V,W, \psi)$ can be regarded as a representation of the morphism Lie algebra $(\mathrm{Lie}(G), \mathrm{Lie}(H), \mathrm{Lie}(\Phi))$. Thus, we can consider the cochain complex $\{ C^\ast_\mathsf{mLA}(\mathrm{Lie}(\Phi), \psi), \delta_\mathsf{mLA} \}$. Following the van Est map (\ref{ve-map}), we are now able to find a connection between these two cochain complexes and induced cohomologies.

\begin{thm}
With the above notations, the collection of maps $\mathrm{vE} : C^n_{\mathsf{mLG}, d} (\Phi, \psi) \rightarrow  C^n_\mathsf{mLA} (\mathrm{Lie}(\Phi), \psi)$ given by
\begin{align*}
(\mathrm{vE}) (\Theta, \Gamma, \Lambda) = \big( (\mathrm{vE}) (\Theta), ~ (\mathrm{vE}) (\Gamma),~ (\mathrm{vE}) (\Lambda)  \big)
\end{align*}
defines a homomorphism of cochain complexes from $\{ C^\ast_{\mathsf{mLG}, d} (\Phi, \psi), \delta_\mathsf{mLG} \}$ to $\{ C^\ast_\mathsf{mLA} (\mathrm{Lie} (\Phi), \psi), \delta_\mathsf{mLA} \}$. Hence there is a induced map $\mathrm{vE} : H^\ast_{\mathsf{mLG}, d} (\Phi, \psi) \rightarrow H^\ast_\mathsf{mLA} (\mathrm{Lie}(\Phi), \psi)$ on the level of cohomology.
\end{thm}

\medskip



\noindent {\bf Acknowledgements.} The author would like to thank IIT Kanpur, India where some parts of the work have been carried out.




\end{document}